\documentclass[11pt]{article}

\usepackage[T1]{fontenc}
\usepackage[utf8]{inputenc}
\usepackage{lmodern}
\usepackage[margin=1in]{geometry}
\usepackage{amsmath,amssymb,amsthm}
\usepackage{microtype}
\usepackage{graphicx,placeins}
\usepackage{tocloft}
\usepackage{xcolor}
\definecolor{linkblue}{RGB}{0,85,170}
\usepackage[colorlinks=true,linkcolor=linkblue,citecolor=linkblue,urlcolor=linkblue]{hyperref}
\hypersetup{
 pdftitle={Elliptic rearrangement for Poisson equations with mean-zero data},
 pdfauthor={Dorin Bucur and Richard S. Laugesen},
 pdfsubject={Paired comparison for Poisson equations with signed mean-zero data},
 pdfkeywords={Poisson equation; mean-zero data; signed rearrangement; Talenti comparison; gradient comparison; clamped plate; MSC 2020: 35J05, 35J25, 35B45}
}

\newtheorem{theorem}{Theorem}
\newtheorem{proposition}[theorem]{Proposition}
\newtheorem{lemma}[theorem]{Lemma}
\newtheorem{conjecture}[theorem]{Conjecture}
\makeatletter
\let\c@table\c@figure
\makeatother
\theoremstyle{remark}
\newtheorem*{remark}{Remark}
\newcommand{\R}{\mathbb{R}}
\newcommand{\dd}{\,dx}
\newcommand{\norm}[1]{\left\lVert#1\right\rVert}

\title{Elliptic rearrangement for Poisson equations with mean-zero data}
\author{
Dorin Bucur
\thanks{
Universit\'e Savoie Mont Blanc, CNRS, LAMA, 73000 Chamb\'ery, France.
\href{mailto:Dorin.Bucur@univ-smb.fr}{\texttt{<Dorin.Bucur@univ-smb.fr>}}
}
\and 
Richard S. Laugesen
\thanks{
Dept. of Mathematics, Univ.\ of Illinois Urbana--Champaign, Urbana, IL 61801, USA.
\href{mailto:laugesen@illinois.edu}{\texttt{<laugesen@illinois.edu>}}
}
}

\begin{document}
\maketitle

\begin{abstract}
The $L^p$ norm of a sign-changing Dirichlet
solution of Poisson's equation is shown to be bounded by the maximum of the $L^p$ norms of the   
ball solutions obtained by signed rearrangement of the source and its
negative, assuming the source term has mean value zero. This paired comparison on bounded open sets extends Talenti’s elliptic rearrangement principle to solutions that change sign. In dimensions one
and two, the paired comparison holds for all $1\le p\le\infty$, and it holds in every dimension for $p=\infty$. The inequalities are sharp and equality cases are characterized. In dimensions $n \geq 3$, paired comparison can fail for every finite $p$ although it remains valid for $p\ge(3n-1)/4$ in the special case of bang-bang sources with mean zero.

The mean-zero source assumption is crucial when $p<\infty$. Without it, paired comparison can fail even for bang-bang sources. For sources of arbitrary mean in one dimension, however, the $L^\infty$ paired comparison still holds and so does a comparison for $L^p$ norms of the derivative when $1\le p\le3$.
\end{abstract}

\begingroup
\renewcommand{\thefootnote}{}
\begin{NoHyper}
\end{NoHyper}
\endgroup

\setcounter{tocdepth}{1}
\setlength{\cftbeforesecskip}{2pt}
\tableofcontents
\clearpage

\section{Comparison theorems for Poisson equations with mean-zero data}

Can a sign-changing potential with mean-zero data be increased in norm by rearranging its source data into a ball? Which symmetric
rearrangement should be used: decreasing or increasing? Our answer is to try both ways and choose the larger of the resulting norms. This ``paired comparison'' succeeds in dimensions
one and two but fails in dimensions three and higher for every $L^p$, except that it survives for $L^\infty$. 
In one dimension, $L^\infty$ comparison holds true even for
arbitrary mean sources, whereas without mean zero the $L^p$
comparison can fail.

To state these results precisely, suppose $n \ge 1$ and consider a nonempty, bounded open set $\Omega\subset\R^n$, and let 
\[
B=\Omega^\ast
\]
be the ball of the same volume centered at the origin. For a measurable function $f: \Omega\to \R$, denote by $f^\ast$ the \emph{symmetric decreasing rearrangement} on $B$. This $f^\ast$ is radially symmetric and satisfies 
\[
 |\{x\in B:f^\ast(x)>\lambda\}| =|\{x\in\Omega:f(x)>\lambda\}| , 
 \qquad \lambda\in\R .
\]
Notice all the values of $f$ are rearranged, including the negative values. 

Suppose $u\in H^1_0(\Omega)$ and $v\in H^1_0(B)$ are weak solutions of Poisson's equation with Dirichlet boundary condition and data $f \in L^2(\Omega)$ and $f^* \in L^2(B)$, respectively: 
\[
 -\Delta u=f\quad\text{in }\Omega,\qquad
 -\Delta v=f^\ast\quad\text{in }B .
\]
Call $f$ and $f^\ast$ the \emph{sources} for $u$ and $v$. Physically, the source represents a charge or mass density generating the solution potential. 

Provided that $u \ge 0$, Talenti's ``elliptic rearrangement'' comparison principle \cite{Talenti1976} (see also Kesavan
\cite{Kesavan2006}) states that rearrangement of the source increases the solution, in the sense that 
\[
u^*(x) \le v(x) , \qquad x \in B ,
\]
and hence each $L^p$ norm of $u$ is less than or equal to the corresponding $L^p$ norm of $v$, that is, $\norm{u}_{L^p(\Omega)}\le \norm{v}_{L^p(B)}$. 

The nonnegativity hypothesis $u \ge 0$ is crucial for Talenti's proof, serving to keep the level sets of $u$ away from the boundary of $\Omega$. Certainly if $f \ge 0$ then $u \ge 0$ by the maximum principle, but there are many sign-changing functions $f$ for which $u$ is nonnegative too. If both $u$ and $f$ change sign then Talenti's $L^p$ norm conclusion can fail. 

This paper aims to discover to what extent $L^p$ comparison of solutions remains valid when $u$ changes sign. The initial insight is that if the source $f$ has mean zero then there seems no reason to prefer $f$ over $-f$, or $u$ over $-u$. And so we apply symmetric decreasing rearrangement to both $f$ and $-f$ in order to determine which choice yields a larger solution in $L^p$. We find that at least one of the resulting solutions has larger or equal norm than the original $u$. Perhaps surprisingly, the winning comparison function can switch depending on $p$. 

\subsection*{Assumptions and results}
\label{sec:results}

In the theorems below, we examine $f\in L^2(\Omega)$ and assume $u\in H^1_0(\Omega)$ and
$v,w \in H^1_0(B)$ are the (unique) weak Dirichlet solutions of  
\begin{equation} \label{eq:foundations}
 -\Delta u=f\quad\text{in }\Omega,\qquad
 -\Delta v=f^\ast\quad\text{in }B ,\qquad
 -\Delta w=(-f)^\ast\quad\text{in }B ,
\end{equation}
so that $v$ has source $f^\ast$ while $w$ has source $(-f)^\ast$. Here $u$ and $f$ are permitted to change sign. 
For Theorems~\ref{thm:planarlp} and~\ref{thm:bangbang},
and for Theorem~\ref{thm:linfty} when $n\ge2$, we further assume the
source has mean value zero,
\[
\int_\Omega f \dd=0 .
\]
Under this mean-zero hypothesis, the comparison potentials
$v$ and $w$ are nonnegative and radially decreasing and have vanishing
normal derivative on $\partial B$, as noted in
Lemma~\ref{lem:comparison} below.
The natural question we address is whether a paired $L^p$ comparison holds for sign-changing solutions with mean-zero sources:  
is it true that 
\[
\int_\Omega |u|^p \dd \le \max \left\{\int_{B} v^p \dd, \int_{B} w^p \dd \right\} ?
\]
\begin{theorem}[One- and two-dimensional paired $L^p$ comparison]\label{thm:planarlp}
For $n=1,2$, suppose $\Omega\subset\R^n$ is a nonempty bounded open set and $B$ is the ball of the same measure.
Assume $f\in L^2(\Omega)$ with $\int_\Omega f\dd=0$. If $u,v,w$ satisfy the assumptions \eqref{eq:foundations} above then 
\begin{equation}\label{eq:lpquestion}
 \norm{u}_{L^p(\Omega)}\le
 \max\left\{\norm{v}_{L^p(B)},\norm{w}_{L^p(B)}\right\} \qquad \text{for } 1\le p\le\infty .
\end{equation}

Now suppose equality holds and $f\not\equiv0$. If $n=2$ and $p>1$, or else if $n=1$, then $\Omega$ is a ball (possibly with a set of  $H^1$-capacity zero removed, when $n=2$) and, after translation, $u=v$ or $u=-w$. If $(n,p)=(2,1)$, then either the same conclusion holds
or $\Omega$ is the union of two disjoint disks, possibly
with a set of zero $H^1$-capacity removed.
In the latter case, after suitable translations, $u$ on
the two disks equals truncated versions of $v$ and $-w$,
respectively, with their boundary values subtracted to make
them vanish on the disk boundaries.

When $n=2$ and $p=1$, the conclusion \eqref{eq:lpquestion} tightens because $v$ and $w$ have the same $L^1$ norm: 
\begin{equation}\label{eq:L1strong}
 \int_\Omega |u|\dd\le\int_Bv\dd=\int_Bw\dd.
\end{equation}

\end{theorem}

The mean-zero hypothesis on the source $f$ is essential
for finite $p$, even in dimension one, as
the explicit bang-bang family in
Proposition~\ref{prop:finitep-arbitrarymean} shows.

\begin{remark}[Both rearrangements are needed]

The norm of $u$ can be greater than the norm of $v$ or $w$ (just not both). An example is presented in Section~\ref{sec:single-source-nonsymmetric} and Figure~\ref{fig:endpointfunctionals}. Hence it is essential that Theorem \ref{thm:planarlp} considers the maximum over $v$ and $w$. Nonetheless, if one knows the nodal volume fraction
$\alpha=|\{u>0\}|/|\Omega|$ then one can identify a comparator
whose norm is guaranteed to bound that of $u$. The criterion is particularly transparent in the planar case with $p=2$ in Section \ref{sec:l2proof}, where a threshold $c$ depending 
only on $f^* \not\equiv 0$ is constructed such that if $\alpha \leq c$ then $u$ has smaller or equal $L^2$-norm than $w$ while if $\alpha \geq c$ then $u$ has smaller or equal norm than $v$.  
\end{remark}

\begin{remark}[The winner can switch with $p$]
For some sources $f$, the solution $v$ coming from $f^*$ has larger
norm at some values of $p$ than the solution $w$ coming from $(-f)^*$, while at other $p$-values the
order is reversed; see Section~\ref{sec:single-source-nonsymmetric}. 
\end{remark}

\begin{remark}[Equality can occur]
Equality can certainly hold in the theorem for nontrivial mean-zero data. For example, suppose $\Omega$ is a ball $B$ and take the sign-changing data to be $f=1$ on the concentric ball of
volume $|\Omega|/2$ and $f=-1$ on the remaining annulus of the same volume. Then
$f=f^\ast=(-f)^\ast$ and hence $u=v=w >0$ in $B$, so that equality holds in the theorem. For generic domains and data, though, one expects inequality \eqref{eq:lpquestion} to hold strictly, including in situations where $u$ does not change sign. 
\end{remark}

Theorem~\ref{thm:planarlp} implies, in our forthcoming work \cite{BucurLaugesenForthcoming}, a new proof of Nadirashvili's theorem that the disk minimizes the fundamental tone of a vibrating clamped plate in two
dimensions (Rayleigh's conjecture); see the Literature section below. More broadly, we
establish in \cite{BucurLaugesenForthcoming} a whole family of nonlinear analogues involving $L^p$ norms of the
plate displacement and $L^q$ norms of its Laplacian.
\begin{theorem}[Paired $L^\infty$ comparison in every dimension]\label{thm:linfty}
For $n \ge 1$, suppose $\Omega\subset\R^n$ is a nonempty bounded open set and $B$ is the ball of the same volume.
Assume $f\in L^2(\Omega)$,
and when $n\ge2$ assume $\int_\Omega f\dd=0$.
If $u,v,w$ are taken to satisfy \eqref{eq:foundations} then
\begin{equation}\label{eq:linftycomparison}
 \norm{u}_{L^\infty(\Omega)}
 \le\max\{\norm{v}_{L^\infty(B)},\norm{w}_{L^\infty(B)}\}.
\end{equation}
Further, when $n=1$, $v$ gives the maximum in \eqref{eq:linftycomparison} if
$\int_\Omega f\dd\ge0$ and $w$ gives the maximum if
$\int_\Omega f\dd\le0$, while both give the maximum if $\int_\Omega f\dd=0$.

If equality holds and $\norm{u}_{L^\infty(\Omega)}$ is finite and positive, then $\Omega$
is a ball (possibly with a set of zero $H^1$-capacity removed, when $n \geq 2$) and furthermore $u=v$ or $u=-w$ after suitably translating $\Omega$. 
\end{theorem}
The mean-zero hypothesis on $f$ is necessary when $n \geq 2$, as one can construct examples with nonzero source mean for which the $L^\infty$ norm of $u$ exceeds the norms of $v$ and $w$. 

The solution $u$ is not, in general, bounded when $f\in L^2(\Omega)$ and $n \geq 4$. Conclusion \eqref{eq:linftycomparison} of the theorem implies that if $u$ is not bounded then either $v$ or $w$ is not bounded; while if both $v$ and $w$ are bounded then $u$ is bounded too. By elliptic regularity, all three solutions are indeed bounded if also $f \in L^q(\Omega)$ for some $q>n/2$, which is automatic when $n=1,2,3$ since $f \in L^2(\Omega)$.

\begin{remark}[$L^p$ comparison up to a factor, in every dimension]
For $n\ge1$ and $1\le p<\infty$, with mean-zero source
the method gives a paired comparison with constant $2^{1/p}$, as proved at the end of Section~\ref{sec:linfty}:
\begin{equation} \label{eq:2p}
 \norm{u}_{L^p(\Omega)}
 \le 2^{1/p} \max\{\norm{v}_{L^p(B)},\norm{w}_{L^p(B)}\}.
\end{equation}
Letting $p\to\infty$ recovers the
$L^\infty$ comparison of \eqref{eq:linftycomparison}. For $n=1,2$, Theorem~\ref{thm:planarlp} and its equality examples
show that the sharp constant is $1$ rather than $2^{1/p}$. For $n\ge3$, the constant $2^{1/p}$ is sharp when
$n/(n-2)\le p<\infty$, as shown by the concentration construction
in Section~\ref{sec:highercounterexamples}: the ratio of the original
norm to the paired comparison norm has lower limit at least
$2^{1/p}$ as $q\downarrow0$, while \eqref{eq:2p} bounds the ratio
above by $2^{1/p}$. We do not know the sharp constant for
$1\le p<n/(n-2)$ when $n \geq 3$, but it is strictly larger than $1$.
\end{remark}

Although Theorem \ref{thm:planarlp} fails for $p<\infty$ in higher dimensions, it does hold on a restricted class of data when $p \geq (3n-1)/4$.

\begin{theorem}[Bang-bang comparison in higher dimensions]
\label{thm:bangbang}
For $n \ge 3$, suppose $\Omega\subset\R^n$ is a nonempty bounded open set and $B$ is the ball of the same volume. Suppose the source $f$ is of bang-bang type with mean zero: for some measurable $E\subset\Omega$ and some
$A_0,B_0>0$,
\begin{equation*}
 f=A_0\mathbf1_E-B_0\mathbf1_{\Omega\setminus E},
 \qquad \int_\Omega f\dd=0.
\end{equation*}
If $u,v,w$ satisfy assumption \eqref{eq:foundations} then 
\begin{equation*}
\begin{split}
\norm{u}_{L^p(\Omega)}
& \le
 \max\left\{\norm{v}_{L^p(B)},\norm{w}_{L^p(B)}\right\} \quad \qquad \text{for\quad} \frac{3n-1}{4}\le p\le\infty , \\
 & = 
 \begin{cases}
\norm v_{L^p(B)} & \text{if } |E|\le|\Omega\setminus E| , \\
\norm w_{L^p(B)} & \text{if } |E|\ge|\Omega\setminus E| .
 \end{cases}
 \end{split}
\end{equation*}
{
Equality holds if and only if $\Omega$ is a ball up to a set of zero
$H^1$-capacity and, after translation,
\[
 \begin{cases}
 u=v& \text{if } |E|<|\Omega\setminus E|,\\
 u=v\ \text{or } u=-w& \text{if } |E|=|\Omega\setminus E|, \\
 u=-w& \text{if } |E|>|\Omega\setminus E|.
 \end{cases}
\]
}
\end{theorem}

The equality identities also determine the geometry of the source: applying $-\Delta$ gives $f=f^*$ when $u=v$ and
$-f=(-f)^*$ when $u=-w$.
Hence $E$ is a centered ball
if $u=v$, while $\Omega\setminus E$ is a centered ball
if $u=-w$, up to sets of measure zero and after the translation.

Does paired comparison extend to gradient norms? In one dimension it
does for $1\le p\le3$, even for sources of arbitrary mean. The sign of the
mean selects the winning comparator. 

\begin{theorem}[One-dimensional gradient $L^p$ comparison]
\label{thm:gradientone}
For $n=1$, suppose $\Omega\subset\R$ is a nonempty bounded open
set and $B=(-R,R)$ is the interval of the same measure. 
Assume $f\in L^2(\Omega)$
and let $u\in H^1_0(\Omega)$ and $v,w\in H^1_0(B)$
be the weak Dirichlet solutions of
\[
 -u''=f\quad\text{in }\Omega,\qquad
 -v''=f^\ast,\quad -w''=(-f)^\ast\quad\text{in }B.
\]
If $1\le p\le3$ then
\[
 \norm{u'}_{L^p(\Omega)}
 \le
 \begin{cases}
 \norm{v'}_{L^p(B)},&\displaystyle\int_\Omega f\dd\ge0,\\[2mm]
 \norm{w'}_{L^p(B)},&\displaystyle\int_\Omega f\dd\le0.
 \end{cases}
\]
When $\int_\Omega f\dd=0$, the two comparison
norms are equal.
\end{theorem}

The theorem is proved in Section~\ref{sec:gradientoneproof}. An equality statement can be developed, but we omit it  in the interests of brevity. 

The corresponding gradient norm inequality can fail in dimensions $n\ge 2$ for $p=2$, even after taking the maximum over $v$ and $w$ on the right side and even if $u$ is required to satisfy the pointwise clamped boundary condition $\partial_\nu u=0$, which is stronger than assuming the source $f$ has mean zero. See the final part of
Section~\ref{sec:gradientoneproof}.

\subsection*{Future direction}
Do the Euclidean results in this paper extend to the sphere and hyperbolic space, and to Gauss space? We hope such possibilities will be explored by other researchers.

\subsection*{Outline of the paper}

Section~\ref{sec:notation} fixes notation for the rearrangements and radial
potentials employed throughout the paper. 
Section~\ref{sec:domainapproximation} establishes the nodal comparison
and its equality rigidity in every dimension. The special case
$(n,p)=(2,2)$ of Theorem~\ref{thm:planarlp} displays the essential
geometry with very little algebra and so we prove it first, in Section~\ref{sec:l2proof}.
Section~\ref{sec:lp}
proves $L^p$ comparison in one and two dimensions
and constructs counterexamples with nonzero source mean.
Section~\ref{sec:dimensionfailure} identifies where the proof breaks down in higher dimensions.
Theorem~\ref{thm:linfty} for $p=\infty$ is proved in
Section~\ref{sec:linfty}, including the arbitrary mean
extension in one dimension. The bang-bang comparison
of Theorem~\ref{thm:bangbang} is proved in
Section~\ref{sec:bangbangproof}. The subsequent two sections construct explicit counterexamples for $p<\infty$ in higher dimensions, and the paper concludes in Section~\ref{sec:gradientoneproof} with the proof of the one-dimensional gradient comparison.

\subsection*{Two-ball literature, and how our work differs}

Our technique is closely related to the \emph{two-ball method}, which has a specific history in 
clamped plate eigenvalue literature. Talenti \cite{Talenti1981}
introduced the strategy of applying his second order elliptic
rearrangement comparison theorem \cite{Talenti1976} separately on the
positive and negative nodal domains of the fourth order bilaplacian eigenfunction for a vibrating plate.
One may use here the signed source form of Talenti's theorem, recorded conveniently in the book of Kesavan \cite{Kesavan2006}. This step transforms the
problem on the original domain into radial comparison problems on two
smaller balls whose volumes add up to that of the original plate. Functions on the 
two balls are linked through a coupling condition inherited by the
boundary fluxes.

Nadirashvili \cite{Nadirashvili1995} sharpened the method to prove
Rayleigh's conjecture in the plane, saying that among all clamped
plates of given area, the first eigenvalue or fundamental tone is minimal for the disk.
Ashbaugh and Benguria \cite{AshbaughBenguria1995} developed the method
further and proved the conjecture in dimensions two and three.
Ashbaugh and Laugesen \cite{AshbaughLaugesen1996} carried the two-ball
analysis into higher dimensions, showing that although it cannot
prove the sharp Rayleigh conjecture, it does yield quantitative lower
bounds whose constants improve toward the sharp value $1$ as the
dimension tends to infinity. Krist\'aly \cite{Kristaly2020,Kristaly2022} extended the two-ball
Rayleigh conjecture method to plates in hyperbolic space and the
sphere, utilizing global signed source
rearrangements. This feature was present already in Talenti's original
plate argument \cite[pp.~268--270]{Talenti1981}; see also Kesavan's
account \cite[Section~4.5]{Kesavan2006} of Rayleigh's plate conjecture. 

After rearranging to the coupled two-ball situation, the method involves minimizing the coupled Rayleigh quotient over all source functions, and then finally optimizing this minimizer with respect to the radii of two balls, wanting the optimum to occur when one ball contains all the volume. 

The paired rearrangement argument in the current paper shares Talenti's nodal reduction
and use of global signed source rearrangements. The subsequent
step is where the methods diverge and our method is simpler: we do not minimize with respect to the source and instead keep it fixed. 

Lastly, as in the two-ball method, we optimize over all possible radii of the balls and aim to show the optimal situation is when one ball has all the volume. But here too a difference arises: because we have not minimized with respect to the source, the two balls are not interchangeable and so our functional is not symmetric. Nonetheless we are able to show in Theorem~\ref{thm:planarlp} for one and two dimensions that the optimum occurs when one ball degenerates and the other
captures all the volume. These endpoint possibilities correspond to rearrangement
of $f$ or $-f$ on the full ball.

To maintain unambiguous terminology in this paper, we reserve the phrase
\emph{two-ball method} for the nodal domain reduction of Talenti
\emph{et al.} followed by source minimization. Our fixed source approach on two balls is called the 
\emph{paired comparison method}.

The only other result we know that gives this particular kind of
single-ball norm comparison for sign-changing Dirichlet solutions is
due to Leylekian \cite[Section~7.1, Corollary~7.4]{Leylekian2024}
(reproduced in \cite[Section~5.5, Theorem~5.2]{ABLL2025}). His result
says that in two dimensions for $u\in H_0^2(\Omega)$, one has
\[
\|u\|_{L^1(\Omega)} \leq \|v\|_{L^1(B)}.
\]
Thus Leylekian obtains the same inequality as
\eqref{eq:L1strong} for $(n,p)=(2,1)$, under the additional clamped hypothesis
$u\in H_0^2(\Omega)$.

\subsection*{Other related literature}

For a broad account of symmetrization techniques, their historical
development, and applications in analysis and partial
differential equations, we refer to Talenti's survey
\emph{The art of rearranging} \cite{Talenti2016}. Also, his earlier paper
\cite{Talenti1987} develops an interesting constructive viewpoint on assembling a
rearrangement.

Optimization over classes of equimeasurable functions was investigated by
Alvino, Trombetti and Lions~\cite{AlvinoLionsTrombetti1989}. Their
viewpoint is close in spirit to the present work, in that the distribution of
the source is fixed while its spatial placement varies. Alvino,
Trombetti and Lions~\cite{LionsAlvinoTrombetti1990} established
broader elliptic and parabolic comparison results through Schwarz
symmetrization, extending the classical framework surrounding
Talenti's theorem. A quite different approach to the subject grew out of Baernstein's
``star function'' technique in complex analysis
\cite{BaernsteinDrasinLaugesen2019}.

Several works in the literature are more directly connected with sign changes. Van den
Berg and Bucur~\cite{VandenBergBucur2020} study Poisson
equations with two-valued sign-changing sources, including the
placement of the two phases and symmetry breaking on balls. Their
setting is especially relevant to the bang-bang sources in
Theorem~\ref{thm:bangbang}. McGillivray~\cite{McGillivray2017} treats a
related two phase membrane problem by rearrangement and isoperimetric
methods. In a different equation and with Neumann boundary conditions,
Salda\~na and Tavares~\cite{SaldanaTavares2018} introduce a
transformation that combines a sign flip with decreasing rearrangement
while preserving the relevant norm, specifically to handle nodal
solutions.

\section{Notation and definitions}
\label{sec:notation}

This section gathers together the signed rearrangements, integrated rearrangements and radial potentials
that are needed for the later proofs. 
We first define signed rearrangements and their integrals for sources of arbitrary mean. The radial potentials and endpoint arguments in most of the paper will assume mean zero.

Fix $n\ge1$. For $R>0$, let $B_R$ denote the open ball of radius
$R$ centered at the origin in $\R^n$, and write $\omega_n=|B_1|$.
Consider a nonempty bounded open set $\Omega$ in $\R^n$. For $f\in L^2(\Omega)$, write
\[
 m=\frac1{|\Omega|}\int_\Omega f\dd
\]
for its mean value. Unless stated otherwise (such as in the one-dimensional arguments in Sections~\ref{sec:linftyone} and~\ref{sec:gradientoneproof}), we normalize $\Omega$ to have the same volume as the unit ball,
\[
 |\Omega|=\omega_n
\]
and write $B=B_1$. Under this normalization, define the nodal volume
fractions
\begin{equation}\label{eq:nodalfractions}
 \alpha=\frac{|\{u>0\}|}{\omega_n},\qquad
 \beta=\frac{|\{u<0\}|}{\omega_n},
\end{equation}
which satisfy $\alpha+\beta\le1$.

\subsection*{Signed rearrangements}

Let $f^\downarrow$ denote the decreasing one-variable rearrangement of $f\in L^2(\Omega)$ on $(0,1)$, with respect to normalized volume $dx/|\Omega|$. When $|\Omega|=\omega_n$, this convention gives 
\begin{equation*}
 f^\ast(x)=f^\downarrow(|x|^n) .
\end{equation*}
The decreasing profile $f^\downarrow$ belongs to $L^2(0,1)$ and under the normalization $|\Omega|=\omega_n$ satisfies
\[
 \omega_n\int_0^1f^\downarrow(t)^2\,dt=\int_\Omega f^2\dd , \qquad 
  \omega_n\int_0^1f^\downarrow(t)\,dt = \int_\Omega f \dd = \omega_n m .
\]
Define the \emph{integrated rearrangement} of $f$ by 
\[
 F(t)=\int_0^t f^\downarrow(s)\,ds,
 \qquad 0\le t\le1.
\]
The decreasing profile makes $F\in C([0,1])$ concave, and hence $F$ lies above the chord joining its endpoint values:
\[
 F(0)=0,\qquad F(1)=m,\qquad F(t)\ge mt,\qquad 0\le t\le1.
\]
The corresponding profile for $(-f)^\ast$ equals $-f^\downarrow(1-t)$, except possibly at a countable number of points where $f^\downarrow$ jumps, and the integrated rearrangement of $-f$ is
\[
 \widetilde F(t)= - \int_0^t f^\downarrow(1-s)\,ds
 =F(1-t)-F(1)=F(1-t)-m.
\]

We shall several times use the Hardy--Littlewood rearrangement inequality in the form
\begin{equation}\label{eq:top}
 \int_E f\dd\le \omega_n F(|E|/\omega_n)
\end{equation}
for measurable $E\subset\Omega$. 

\subsection*{Radial potentials}

Throughout most of the paper we assume the source $f$ has mean zero, so that $m=0$. Then 
\[
 F(0)=F(1)=0,\qquad F(t)\ge0,\qquad
 \widetilde F(t)=F(1-t),\qquad 0\le t\le1.
\]
If $f\not\equiv0$ then concavity and the zero endpoint values imply $F(t)>0$ for $0<t<1$.

For $0\le a,b\le1$, define nonnegative decreasing functions by 
\begin{align}
 V_a(t)&=
 \begin{cases}
n^{-2} \displaystyle{\int_t^a
 F(s)s^{2/n-2}\,ds},&0<t \leq a,\\[4pt]
 0,&a\le t\le1,
 \end{cases}\label{eq:Vanhigher}\\
 W_b(t)&=
 \begin{cases}
 n^{-2} \displaystyle{\int_t^b
 \widetilde F(s)s^{2/n-2}\,ds},&0<t \leq b,\\[4pt]
 0,&b\le t\le1.
 \end{cases}\label{eq:Wbnhigher}
\end{align}
The value at $t=0$ is understood as a limit and may be infinite. Note these functions are supported on $[0,a]$ and $[0,b]$, respectively. 

The functions $v$ and $w$ used in the paired comparison theorems of this paper arise when $a=1$ and $b=1$, that is, the full-radius case. Specifically, define nonnegative radially symmetric decreasing functions
\begin{equation} \label{eq:vwdef}
 v(x)=V_1(|x|^n),\qquad
 w(x)=W_1(|x|^n),\qquad x\in B.
\end{equation}
\begin{lemma}[Comparison potentials on the ball]\label{lem:comparison}
If $f$ has mean zero, $m=0$, then the functions $v$ and $w$ in \eqref{eq:vwdef} belong to $H^1_0(B)$ and are weak solutions of  
\[
 -\Delta v=f^\ast\quad\text{in }B ,\qquad
 -\Delta w=(-f)^\ast\quad\text{in }B .
\]
They are nonnegative, radially decreasing, $C^1$-smooth away from the origin, and on $\partial B$ they vanish and have normal derivative zero. If $f\not\equiv0$ then $v$ and $w$ are positive in $B$ and their radial
derivatives are negative for $0<r<1$.
\end{lemma}
We sketch the proof for completeness.
\begin{proof}
It suffices to consider $v$, since the argument for \(w\) is analogous.
The sign and monotonicity assertions follow directly from
\eqref{eq:Vanhigher} since $F\ge0$, with $F>0$ on $(0,1)$ when
$f\not\equiv0$.

The $C^1$-smoothness of $v$ away from the origin is clear because continuity of $F(s)$ for $0 \leq s \leq 1$ ensures continuity of $V_1^\prime(t)$ for $0<t \leq 1$. Further, $V_1(1)=0$ and so $v=0$ on $\partial B$, while $V_1^\prime(1)=0$ because $F(1)=0$ and so $v$ has normal derivative zero at the boundary $\partial B$. 

Next we verify \(v\in L^2(B)\). Cauchy--Schwarz gives $|F(s)| \le s^{1/2}\norm{f^\downarrow}_{L^2(0,1)}$ and so 
\[
 |v(r)|
 \le \norm{f^\downarrow}_{L^2(0,1)} n^{-2} 
      \int_{r^n}^1s^{2/n-3/2}\,ds
 \le C_n\norm{f^\downarrow}_{L^2(0,1)}
 \begin{cases}
  1, & n\le3,\\[2mm]
  \log(1/r), & n=4,\\[2mm]
  r^{2-n/2}, & n\ge5.
 \end{cases}
\]
Here $C_n$ depends only on $n$. Each bound is square integrable with respect to \(r^{n-1}\,dr\) near
the origin, and so \(v\in L^2(B)\).

To show $v \in H^1_0(B)$, first observe $ v_r=-F(r^n)/nr^{n-1}$, so that spherical coordinates and the substitution \(t=r^n\) yield
\[
 \int_B|\nabla v|^2\dd
 =\frac{\omega_n}{n^2}
   \int_0^1F(t)^2t^{2/n-2}\,dt < \infty ,
\]
by recalling that $F(t) \leq C t^{1/2}$. A truncation argument around the origin now shows that \(v\in H^1_0(B)\). 

Finally, $- r^{1-n} \bigl(r^{n-1}v_r(r)\bigr)' = F'(r^n)=f^\downarrow(r^n)$ for almost every \(r>0\), that is, $-\Delta v = f^*$ pointwise. It is now straightforward to show $-\Delta v = f^*$ weakly in the $H^1_0$-sense. 
\end{proof}

\subsection*{$J$-functional for $L^p$ norms}

When $1 \leq p < \infty$, we measure the $L^p$ norms of the potentials with the functional 
\begin{equation}\label{eq:Jnp}
 J_{n,p}[F](a)=
 \int_0^1 V_a(t)^p\,dt
 +\int_0^1 W_{1-a}(t)^p\,dt,
 \qquad 0\le a\le1.
\end{equation}
Here $V_a$ and $W_{1-a}$ are constructed from $F$. When $F$ is fixed, we abbreviate $J_{n,p}[F](a)$ to $J_{n,p}(a)$ and further to $J_p(a)$ when the dimension $n$ is fixed.

\subsection*{Rescaled profiles and moments}

For $0<a<1$, the rescaled profiles on $(0,1)$ are
\begin{equation}\label{eq:rescaledprofiles}
 P_a(t)=n^2V_a(at),\qquad
 Q_a(t)=n^2W_{1-a}((1-a)t),\qquad 0<t<1.
\end{equation}
For a nonnegative function $Y$ on $(0,1)$, set
\begin{equation}\label{eq:profilemoments}
 M_0(Y)=1,\qquad
 M_r(Y)=\int_0^1Y(t)^r\,dt , \quad r \neq 0,
\end{equation}
with the integral required to be finite. For $p>1$, define the moment difference
\begin{equation}\label{eq:momentdifference}
 D_{n,p}(a)=a^{2/n-1}M_{p-1}(P_a)
           -(1-a)^{2/n-1}M_{p-1}(Q_a),\qquad 0<a<1,
\end{equation}
whenever both moments are finite. When the dimension $n$ is fixed, we abbreviate $D_{n,p}$ to $D_p$.

\section{Nodal comparison in all dimensions}
\label{sec:domainapproximation}

This section develops the common part of the proofs in every
dimension $n\ge1$: Talenti comparison on the two nodal sets, rigidity
when that comparison is saturated, and the reduction of $L^p$
norm comparisons to an endpoint inequality. The later sections then prove the main theorems by 
resolving the dimensions and parameters for which that
endpoint inequality holds.

We work with $u,v,w$ as assumed in Section~\ref{sec:results},
with mean-zero source $f$, and rescale without loss of generality to $|\Omega|=\omega_n$ so that $B=B_1$. We start by comparing the positive and negative parts separately.
Put
\[
 u_+=\max(u,0),\qquad u_-=\max(-u,0),
\]
and recall their nodal volume fractions $\alpha,\beta$ from
\eqref{eq:nodalfractions}. Let $U_+$ and $U_-$ denote the decreasing (one-dimensional) 
rearrangements of $u_+$ and $u_-$ on $(0,1)$. 
These functions vanish after $t=\alpha$ and $t=\beta$, respectively, and
for $1\le p<\infty$ we have
\begin{equation*}
 \int_\Omega |u|^p\dd
 =\omega_n\int_0^\alpha U_+(t)^p\,dt
  +\omega_n\int_0^\beta U_-(t)^p\,dt,
\end{equation*}
with the integrals understood in the extended sense.

To better visualize the construction, let $u^\downarrow$ be the signed
decreasing rearrangement of the solution $u$ with respect to normalized
volume, so that 
\[
 u^\downarrow(t)
 =\inf\Bigl\{\lambda\in\R: \frac{1}{|\Omega|} 
 |\{x\in\Omega:u(x)>\lambda\}|\le t\Bigr\} , \qquad 0<t<1 .
\]
Then 
\[
 U_+(t)=\max\{u^\downarrow(t),0\},\qquad
 U_-(t)=\max\{-u^\downarrow(1-t),0\} ,
\]
for almost every $t$. That is, the negative tail is reflected and its sign reversed, as shown
in Figure~\ref{fig:signedprofiles}. No symmetry of $u^\downarrow$ is assumed.

\begin{figure}[htbp]
 \centering
 \includegraphics[width=0.92\linewidth]{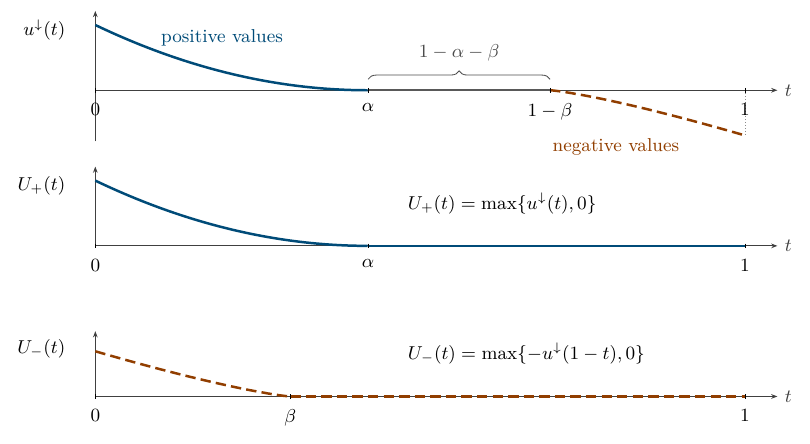}
 \caption{Separating the two signs of the solution (schematic).
 The signed decreasing rearrangement $u^\downarrow$ lists the values
 of $u$ from largest to smallest. Its positive part gives $U_+$;
 reflecting the negative tail and changing its sign gives $U_-$.
 All three horizontal
 axes show normalized volume $t$. The zero plateau of $u^\downarrow$ has length
 $1-\alpha-\beta$; the two nonnegative profiles have support lengths
 $\alpha$ and $\beta$.}
 \label{fig:signedprofiles}
\end{figure}

Functions $V_a$ and $W_b$ were defined in \eqref{eq:Vanhigher}--\eqref{eq:Wbnhigher}. Remember that $a,b$ are free parameters, whereas $\alpha,\beta$ are the nodal fractions determined by $u$ and hence are fixed.

\begin{lemma}[Talenti theorem on the two nodal sets]\label{lem:level}
For $n\ge1$, under the hypotheses above one has
\begin{align*}
 U_+(t)&\le V_\alpha(t), && 0<t<\alpha ,\\
 U_-(t)&\le W_\beta(t), && 0<t<\beta .
\end{align*}
\end{lemma}

\begin{proof}
It suffices to prove the inequality for $U_+$, since the argument
for $U_-$ is analogous. If $\alpha=0$, there is nothing to prove.
Suppose $\alpha>0$ and put
\[
 \Omega_+=\{u>0\},\qquad g=f|_{\Omega_+}.
\]
To clarify ideas, suppose $u$ has a continuous representative on $\Omega$.
Then $\Omega_+$ is open, and the standard Sobolev restriction
property gives $u|_{\Omega_+}\in H^1_0(\Omega_+)$, satisfying
$-\Delta u=g$ weakly there.

Let $g^\downarrow$ denote the decreasing rearrangement of $g$
on $(0,\alpha)$, using normalized volume, and define
\[
 G(s)=\int_0^s g^\downarrow(r)\,dr,
 \qquad 0\le s\le\alpha.
\]
The measurable set characterization of rearrangement and the
Hardy--Littlewood inequality \eqref{eq:top} give
\[
 G(s)=\frac1{\omega_n}
 \sup_{\substack{E\subset\Omega_+\\ |E|=\omega_n s}}
 \int_E g\,dx
 \le F(s),\qquad 0<s<\alpha.
\]
Talenti's comparison theorem for a nonnegative weak solution
with a signed source
\cite[Theorem~3.1.1, p.~48]{Kesavan2006}
(see also \cite{Talenti1976}) therefore yields
\[
 U_+(t)
 \le\frac1{n^2}\int_t^\alpha G(s)s^{2/n-2}\,ds
 \le\frac1{n^2}\int_t^\alpha F(s)s^{2/n-2}\,ds
 =V_\alpha(t),\qquad 0<t<\alpha.
\]

Without continuity of $u$, the same bound follows from the
positive level argument in Kesavan's proof
\cite[Theorem~3.1.1]{Kesavan2006}, using \eqref{eq:top} and the
distribution function $\mu(\tau)=|\{u>\tau\}|$, with
$\mu(0^+)=\omega_n\alpha$. The proof uses the truncations
$\min\{1,(u-\tau)_+/h\}\in H^1_0(\Omega)$ for $\tau,h>0$
as test functions in the weak equation on $\Omega$.
Thus no openness of the nodal sets is required.
\end{proof}

The next lemma is important for the equality statements in
Theorems~\ref{thm:planarlp}--\ref{thm:bangbang} but can be skipped
by readers who are focussed on the norm inequalities.

\begin{lemma}[Rigidity of a saturated nodal comparison]
\label{lem:nodalrigidity}
For \(n\ge1\), under the assumptions of Section \ref{sec:results},
if $f\not\equiv0$, $\alpha>0$ and $U_+=V_\alpha$ 
then the zero extension of $u_+$ is a translate of the zero extension
of $x\mapsto V_\alpha(|x|^n)$.
If $u$ is continuous in $\Omega$, then the nodal set
$\{u>0\}$ is the corresponding ball, possibly with a set of zero
$H^1$-capacity removed when $n\ge2$. 
If also $\alpha=1$, then $\Omega$ is a
ball (possibly with a set of zero $H^1$-capacity removed, when $n \geq 2$). The analogous assertions hold
for $u_-$ and $W_\beta$.
\end{lemma}

\begin{proof}
Put $h=u_+$. By the hypothesis $U_+=V_\alpha$, we have $h^*(x)=U_+(|x|^n)=V_\alpha(|x|^n)$. Hence $h^*(x)=v(x)-V_1(\alpha)$ on the ball $B_{\alpha^{1/n}}$ and $h^*$ equals $0$ outside that ball. 

Testing the equation for $u$ by $h$, and using the 
P\'olya--Szeg\H{o} and Hardy--Littlewood inequalities, gives that 
\[
 \int_B|\nabla h^*|^2\dd
 \le\int_\Omega|\nabla h|^2\dd
 =\int_\Omega fh\dd
 \le\int_B f^*h^*\dd
 =\int_B|\nabla h^*|^2\dd;
\]
here the last identity follows by testing the radial Poisson equation on
$B_{\alpha^{1/n}}$ by its Dirichlet potential $h^*$. Thus equality
holds in the P\'olya--Szeg\H{o} inequality. Since $F(t)>0$ on $(0,1)$, we have 
$|\nabla h^*|>0$ wherever
$0<h^*<\operatorname*{ess\,sup}h$.
The Brothers--Ziemer rigidity theorem, in the form stated and proved
by Ferone--Volpicelli \cite[Theorem~1.1]{FeroneVolpicelli2003},
therefore gives that $h$ is a translate of $h^*$, as desired.
This is also the alternative indicated by Kesavan
\cite[Remark~3.2.2, p.~60]{Kesavan2006}; the energy argument displayed above 
is the needed adaptation of his result to signed $f$.

By translating $\Omega$ and $h$, we may suppose $h=h^*$ almost
everywhere. Put $D=B_{\alpha^{1/n}}$.
The zero extension of $h\in H^1_0(\Omega)$ vanishes
quasi-everywhere outside $\Omega$. Since $h=h^*$ almost
everywhere, their quasi-continuous representatives agree
quasi-everywhere. Positivity of the radial representative in $D$
therefore implies that $D\setminus\Omega$
has zero $H^1$-capacity. In dimension one this set is empty.
If $u$ is continuous in $\Omega$, equality with the radial profile
gives $\{u>0\}=\Omega\cap D$, proving the nodal set assertion.
If $\alpha=1$, then $D=B$, and $h=h^*$ together with
$|\Omega|=|B|$ gives $|\Omega\setminus B|=0$.
Openness then implies $\Omega\subset B$, proving the domain assertion.

The argument is analogous for $u_-$ and $W_\beta$. 
\end{proof}

\subsection*{Reduction to an endpoint inequality}
The proofs of the main theorems will first invoke the Talenti-based rearrangement from Lemma \ref{lem:level}, which applies in all dimensions, and then for each situation an endpoint inequality will be needed on $J_{n,p}$, as follows. 

Suppose $p<\infty$. Raise the bounds in Lemma \ref{lem:level} to the power $p$ and integrate. Using that 
$\beta\le1-\alpha$ and so $W_\beta \leq W_{1-\alpha}$, one finds 
\begin{equation}\label{eq:dimensionpbound}
 \begin{split}
 \int_\Omega |u|^p\dd
 &\le\omega_n\left[
    \int_0^1V_\alpha(t)^p\,dt
   +\int_0^1W_\beta(t)^p\,dt\right]\\
 &\le\omega_n J_{n,p}(\alpha) .
 \end{split}
\end{equation}
Since $V_1$ and $W_1$ are the radial profiles of $v$ and $w$,
the endpoint values of $J_{n,p}(a)$ satisfy
\begin{equation}\label{eq:comparisonendpoints}
 \omega_n J_{n,p}(1)=\int_Bv^p\dd,
 \qquad
 \omega_n J_{n,p}(0)=\int_Bw^p\dd.
\end{equation}
(In higher dimensions these integrals may be infinite, in which case the inequalities
and identities are understood in the extended sense.)
Hence the norm comparison inequality of type \eqref{eq:lpquestion} will follow provided we establish an endpoint inequality  
\[
 J_{n,p}(a)\le\max\{J_{n,p}(0),J_{n,p}(1)\},
 \qquad 0\le a\le1.
\]
For the equality statements we will work with Lemma \ref{lem:nodalrigidity}.  

The next two sections establish endpoint inequalities
for $n=1,2$, initially with $p=2$ and then with general $p$. Section~\ref{sec:bangbangproof} develops an endpoint inequality for bang-bang sources.

\FloatBarrier
{\section{Proof of Theorem~\ref{thm:planarlp} in dimension two for
\texorpdfstring{$p=2$}{p=2}}
\label{sec:l2proof}}

The $L^2$ endpoint argument in the plane 
exposes the geometry behind the paired comparison with rather few complications, and so in this section we specialize to $n=2$ and $p=2$.

Lemma~\ref{lem:endpoint} below is the central step in this planar $L^2$ argument. It says that assigning all
the area to one disk gives the largest estimate on $J_{2,2}$. We emphasize again that the source and its integrated rearrangement $F$ are
held fixed while the disk areas are varied, and we do not assume the concave function $F$ to be symmetric about $t=1/2$. The lemma involves only properties of $F$ and so we state it in those terms. 

To prove the inequality in Theorem~\ref{thm:planarlp} for $(n,p)=(2,2)$, simply combine Lemma~\ref{lem:endpoint} with
\eqref{eq:dimensionpbound}--\eqref{eq:comparisonendpoints}. The equality characterization is proved later in
Section~\ref{sec:lp}.

For simplicity we write
$J=J_{2,2}$ in the lemma, so that  
\[
 J(a)=\int_0^1V_a(t)^2\,dt + \int_0^1W_{1-a}(t)^2\,dt .
\]
Due to the compact supports, these integrals run over only $0<t<a$ and $0<t<1-a$, respectively. 

\begin{lemma}[Endpoint maximum for the $L^2$ functional]\label{lem:endpoint}
If $F\in C([0,1])$ is concave with $F(0)=F(1)=0$ then $J$ is continuous on $[0,1]$ and
\[
 J(a)\le\max\{J(0),J(1)\} \qquad \text{for } 0\le a\le1.
\]
If $F\not\equiv0$ then the inequality is strict, with $J(a)<\max\{J(0),J(1)\}$ for $0<a<1$.
\end{lemma}

\begin{proof}
Letting
\[
 A(a)=\int_0^aF(s)\,ds,
\]
we shall use the identity 
\begin{equation}\label{eq:tonelli}
\int_0^1V_a(t)^2\,dt
 =\frac18\int_0^a\frac{F(s)A(s)}s\,ds ,
\end{equation}
which follows by expanding the square in the definition of $V_a$ in \eqref{eq:Vanhigher} with $n=2$ to produce a double
integral with kernel $\min(r,s)/(rs)$, and then splitting across $r=s$. Note that the integrand in \eqref{eq:tonelli} is bounded
by $\norm{F}_\infty^2$, and so this identity proves continuity of the left side even at
$a=0$. It follows from \eqref{eq:tonelli} applied to $F(t)$ and analogously to $\widetilde{F}(t)=F(1-t)$ that
$J(a)$ is continuous on $[0,1]$, proving the first claim in the lemma. Note for use below that $\widetilde{A}(1-a)=A(1)-A(a)$. 

For $0<a<1$, differentiation of \eqref{eq:tonelli} yields that 
\begin{equation*}
 J'(a)=\frac{F(a)}{8a(1-a)} H(a)
\end{equation*}
where $H(a)=A(a)-aA(1)$. Thus the slope of $J$ is controlled by the sign of $H$. 

If $F\equiv0$ then $A \equiv H \equiv 0$ and so $J$ is constant and the lemma is immediate. Otherwise $F(a)>0$ on $(0,1)$ by concavity, and so $A(1)>0$. Note that 
\[
 H'(a)=F(a)-A(1),\qquad H(0)=H(1)=0.
\]
Since the maximum value of $F$ strictly exceeds its average value $A(1)$, the concavity and zero endpoint values of $F$ imply that the set $\{a:F(a)>A(1)\}$ is an interval
$(a_-,a_+)$ with $0<a_-<a_+<1$, and that $F<A(1)$ on $(0,a_-)$ and on $(a_+,1)$. Hence $H$ is strictly decreasing on $[0,a_-]$, strictly increasing on
$(a_-,a_+)$, and strictly decreasing on $[a_+,1]$. Therefore $H$ has exactly one interior zero point $c$, occurring between $a_-$ and $a_+$, and
\[
 H(a)<0\quad \text{on } (0,c),\qquad H(a)>0\quad \text{on } (c,1).
\]
Recalling that $H$ determines the sign of $J^\prime$, we conclude that $J$ strictly decreases up until $c$ and
strictly increases thereafter and so $ J(a) < \max\{J(0),J(1)\}$ for all $0<a<1$. 
\end{proof}

\section{Proof of Theorem~\ref{thm:planarlp} in the general case}
\label{sec:lp}

The proof for general $p$ when $n=2$ will be broken into three parts: $p=1$, $1<p\le2$, and $2\le p<\infty$. This overlap at $p=2$ is deliberate, because the $L^2$
case provides a useful meeting point for two rather different
monotonicity arguments. The first case $p=1$ is special because the contributions fit together exactly. For $1 < p \leq 2$, a difference of
moments is shown to be monotone, while above $p=2$ a ratio of moments takes over.

The one-dimensional case ($n=1$) is easier and will be treated later in the section. 

In both dimensions one and two, the endpoint $p=\infty$ is supplied by Theorem~\ref{thm:linfty}, whose proof for all dimensions is given in Section~\ref{sec:linfty}.

\subsection{The \texorpdfstring{$L^1$}{L1} case, \texorpdfstring{$n=2$}{n=2}}
\label{sec:planarlone}

For $p=1$, no endpoint
optimization is needed. Lemma~\ref{lem:level}, Tonelli's theorem and
$\alpha+\beta\le1$ give that 
\[
 \int_\Omega |u|\dd
 \le\pi\int_0^1\bigl(V_\alpha(t)+W_\beta(t)\bigr)\,dt
 =\frac{\pi}{4}\left[
       \int_0^\alpha F(s)\,ds
      +\int_0^\beta F(1-s)\,ds\right]
 \le\frac{\pi}{4}\int_0^1F(s)\,ds.
\]
Since $V_1$ and $W_1$ are the profiles of $v$ and $w$, the final
quantity equals both $\int_Bv\dd$ and $\int_Bw\dd$.

\begin{remark}
Equality can reign even when the solution changes sign.
Indeed, take $0<q<R=2^{-1/2}$ in the sign-changing construction of
Section~\ref{sec:highercounterexamples}; that construction is developed 
for higher dimensional counterexamples but it applies
unchanged for $n=2$. One obtains
\[
 F(t)=\min\{t,q^2,1-t\},\qquad
 \int_\Omega |u|\dd
 =\int_Bv\dd=\int_Bw\dd
 =\frac{\pi}{4}(q^2-q^4).
\]
\end{remark}

\subsection{Endpoint lemma for the \texorpdfstring{$L^p$}{L^p} functional when \texorpdfstring{$1 < p \leq 2, n=2$}{1 < p <= 2, n=2}}

Lemma~\ref{lem:lpendpoint} below extends the planar comparison from $L^2$ to every $L^p$
with $1<p \leq 2$. Again let $F\in C([0,1])$ be nonnegative and concave, with
$F(0)=F(1)=0$. 

The functional $J_p=J_{2,p}$ defined by
\eqref{eq:Jnp} with $n=2$ is 
\begin{equation}
 J_p(a)
 = \int_0^a V_a(t)^p \, dt + \int_0^{1-a} W_{1-a}(t)^p \, dt , \qquad 0 \leq a \leq 1 . \label{eq:Jporig} 
\end{equation}
Its endpoint values are defined by continuity, as justified by the
logarithmic bounds in the proof below. 

\begin{lemma}[Endpoint comparison for $1<p\le2$]\label{lem:lpendpoint}
When $1<p\le2$ we have 
\[
 J_p(a)\le  \max\{J_p(0),J_p(1)\} 
 \qquad \text{for } 0\le a\le1.
\]
If $F\not\equiv0$ then the inequality is strict for every $0<a<1$.
\end{lemma}

\begin{proof}
The case $F\equiv0$ is immediate. Assume otherwise, so that $F(a)>0$ for $0<a<1$.

We use $P_a,Q_a$ from \eqref{eq:rescaledprofiles} with $n=2$ and write $D_p=D_{2,p}$ as in \eqref{eq:momentdifference}. For $0<a,t<1$, these profile formulas are
\begin{equation}\label{eq:planarprofiles}
 \begin{aligned}
 P_a(t)&=\int_t^1\frac{F(as)}s\,ds=4V_a(at),\\
 Q_a(t)&=\int_t^1\frac{F(1-(1-a)s)}s\,ds
       =4W_{1-a}((1-a)t).
 \end{aligned}
\end{equation}
The moment difference specializes to
\begin{equation} \label{eq:momentdifference2dim}
 D_p(a)=\int_0^1(P_a^{p-1}-Q_a^{p-1})\,dt.
\end{equation}
Differentiating \eqref{eq:Jporig}, with technical justification given below, and recalling that $V_a(a)=W_{1-a}(1-a)=0$, one finds 
\begin{equation}\label{eq:Jpderivative}
 J_p'(a)=\frac{p}{4^p}F(a)D_p(a).
\end{equation}
We will show $D_p$ is increasing.

First, we compute the $a$-derivatives from $4V_a(at)$ and $4W_{1-a}((1-a)t)$, getting 
\begin{equation} \label{eq:aderivs}
 \partial_aP_a(t)=\frac{F(a)-F(at)}a,\qquad
 \partial_aQ_a(t)=\frac{F(1-(1-a)t)-F(a)}{1-a}.
\end{equation}
Next, integration by parts in $t$ gives
\[
 \int_0^1P_a(t)^{p-1}\,dt
 =(p-1)\int_0^1P_a(t)^{p-2}F(at)\,dt,
\]
where the boundary terms vanish because $P_a(1)=0$ and $P_a(t)$ grows at most logarithmically as $t \to 0$, by below. The corresponding identity holds for $Q_a$ with $F(1-(1-a)t)$. Using these formulas, we compute that 
\begin{equation}\label{eq:Deta}
 \begin{split}
 a(1-a)D_p'(a)
 ={}&(p-1) F(a)
 \left[(1-a)\int_0^1P_a^{p-2}\,dt
            +a\int_0^1Q_a^{p-2}\,dt\right]\\
 &-\left[(1-a)\int_0^1P_a^{p-1}\,dt
            +a\int_0^1Q_a^{p-1}\,dt\right].
 \end{split}
\end{equation}

Now we analyze the right side of \eqref{eq:Deta}. Concavity of $F$ and the trivial identity
$(1-a)as+a(1-(1-a)s)=a$ imply that 
\[
 (1-a)F(as)+aF(1-(1-a)s)\le F(a).
\]
Dividing by $s$ and integrating over $t<s<1$ bounds the linear combination
\begin{equation*}
 \mathcal L_a(t):=(1-a)P_a(t)+aQ_a(t)\le F(a)\log(1/t).
\end{equation*}
Since $z \mapsto z^{p-1}$ is concave and increasing for $1<p\le2$, we deduce
\[
 (1-a)\int_0^1P_a^{p-1}\,dt+a\int_0^1Q_a^{p-1}\,dt
 \le\int_0^1\mathcal L_a^{p-1}\,dt
 \le F(a)^{p-1}\Gamma(p).
\]
Similarly, since $z \mapsto z^{p-2}$ is convex and decreasing, one has 
\[
 (1-a)\int_0^1P_a^{p-2}\,dt+a\int_0^1Q_a^{p-2}\,dt
 \ge\int_0^1\mathcal L_a^{p-2}\,dt
 \ge F(a)^{p-2}\Gamma(p-1).
\]
Substitution into \eqref{eq:Deta}, together with the functional equation 
$(p-1)\Gamma(p-1)=\Gamma(p)$, proves that $D_p'(a)\ge0$.

Incidentally, when $p=2$, formula~\eqref{eq:Deta}
reduces to
\[
 a(1-a)D_2'(a)
 =F(a)-\int_0^1\mathcal L_a(t)\,dt,
\]
which is immediately nonnegative since
$\mathcal L_a(t)\leq F(a)\log(1/t)$ and $\int_0^1\log(1/t)\,dt=1$.

To justify the differentiations through the integral and integrations by parts in the preceding argument, consider $a$
in a compact subinterval of $(0,1)$ and rescale both integrals in \eqref{eq:Jporig} to $(0,1)$ to get that 
\begin{equation*}
 J_p(a)
=\frac1{4^p}\left[
   a\int_0^1P_a(t)^p\,dt
  +(1-a)\int_0^1Q_a(t)^p\,dt\right] , \qquad 0<a<1 .
\end{equation*}
Near $t=0$, concavity and the zero endpoints give that $F(as) \geq s F(a)$ and $F(1-(1-a)s) \geq s F(a)$. Hence $P_a(t),Q_a(t)\ge F(a)(1-t)$, which controls these functions to the power $p-2$ near
$t=1$. The upper bound
$P_a,Q_a\le\norm{F}_\infty\log(1/t)$ controls the nonnegative
powers and boundary terms used above.

Now we can prove the strict endpoint comparison for $J_p(a)$. As $a\downarrow0$, $P_a\to0$ in $L^{p-1}(0,1)$ and $Q_a$
converges to a nonzero function. Thus $D_p(a)$ has a negative
limit at $0$. Its limit at $1$ is positive by an analogous argument.
Since $D_p$ is increasing, \eqref{eq:Jpderivative} shows that
$J_p(a)$ decreases (and does so strictly for $a$ near $0$) and then increases (strictly for $a$ near $1$), with a possible
constant interval between. Its maximum is therefore attained at an endpoint, and every interior value lies strictly below the larger of the two endpoint
values. Continuity of $J(a)$ at the endpoints $a=0,1$ follows
by dominated convergence, using the logarithmic upper bound.
\end{proof}

\subsection{Moment ratios and endpoint lemma for \texorpdfstring{$2\le p<\infty, n=2$}{2 <= p < infinity, n=2}}

For $p\ge2$, the moment difference $D_p(a)$ need not be monotone. The ratio of the two moments is more
accommodating, and a supporting line to $F$ will supply the comparison. The case $F\equiv0$ is immediate and so in the lemma below we assume $F\not\equiv0$. The profiles $P_a,Q_a$ are given by \eqref{eq:planarprofiles} and $D_p=D_{2,p}$ is the moment difference for $n=2$ in \eqref{eq:momentdifference2dim}.

\begin{lemma}[Strict monotonicity of moment ratios]\label{lem:largelp}
Let $F\in C([0,1])$ be nonnegative, concave, and not identically zero, with
$F(0)=F(1)=0$. Then for each fixed $q\ge1$, the ratio
\begin{equation*}
 R_F(a)=\frac{M_q(P_a)}{M_q(Q_a)},
\end{equation*}
is strictly increasing on $(0,1)$, from $0$ to $+\infty$.
Consequently, $J_p(a)$ strictly
decreases to a unique minimum and then strictly increases, when $2\le p<\infty$ is fixed, and hence 
\[
 J_p(a)<\max\{J_p(0),J_p(1)\}
 \qquad \text{for } 0<a<1 .
\]
\end{lemma}

\begin{proof}
Fix $q\ge1$. Write $b=1-a$ and $c=F(a)>0$.
By the $a$-derivative formulas in \eqref{eq:aderivs} and integration by parts, one has 
\[
 \partial_a P_a(t)=\frac{c-F(at)}a,\qquad
 \partial_a Q_a(t)=\frac{F(1-bt)-c}b,
\]
\[
 M_q(P_a)=q\int_0^1P_a(t)^{q-1}F(at)\,dt,
 \qquad
 M_q(Q_a)=q\int_0^1Q_a(t)^{q-1}F(1-bt)\,dt,
\]
where the boundary terms vanish because $P_a$ and $Q_a$ vanish at $t=1$ and grow at most logarithmically as $t \to 0$,  as justified at the end of this proof. Thus one finds
\begin{equation}\label{eq:logrderivative}
 ab\,\frac{d}{da}\log R_F(a)
 =qc\left[
 b\,\frac{M_{q-1}(P_a)}{M_q(P_a)}
 +a\,\frac{M_{q-1}(Q_a)}{M_q(Q_a)}
 \right]-1.
\end{equation}
For $q=1$, recall $M_0(P_a)=M_0(Q_a)=1$ by convention. We will bound the right-hand side from below by replacing $F$ with a supporting
line at $a$.

First we record a moment comparison for the supporting line argument.
Suppose $q\ge1$ and $0<Y\le Z$ on $(0,1)$, with $Z$ decreasing,
$Y/Z$ increasing and both functions having finite $q$-th moments.
The following argument shows that
\begin{equation}\label{eq:momentquotientorder}
 \frac{M_q(Y)}{M_{q-1}(Y)}
 \le\frac{M_q(Z)}{M_{q-1}(Z)}.
\end{equation}
Put $r=Y/Z$, so that $0<r\le1$ and $r^{q-1}$ is increasing.
The pointwise bound $Y^q\le Y^{q-1}Z$ gives the first inequality
below. The second follows from the reversed weighted Chebyshev
integral inequality, applied to the increasing function $r^{q-1}$
and the decreasing function $Z$, with weight $Z^{q-1}$:
\[
 \frac{M_q(Y)}{M_{q-1}(Y)}
 \le\frac{\int_0^1 r^{q-1}Z^q\,dt}
          {\int_0^1 r^{q-1}Z^{q-1}\,dt}
 \le\frac{M_q(Z)}{M_{q-1}(Z)}.
\]
When $q=1$, the second inequality is equality.
If $Y<Z$ on a set of positive measure, then the first inequality
is strict.

Choose a supporting line at $a$ for the concave function $F$, satisfying 
\[
 L(t)=c+\sigma(t-a)\ge F(t),\qquad
 d_0=L(0)=c-\sigma a\ge0,\quad d_1=L(1)=c+\sigma b\ge0.
\]
Let $\overline P_a,\overline Q_a$ be obtained by replacing $F$ with $L$ in the planar formulas \eqref{eq:planarprofiles}. The following chord argument shows that $F/L$ is
increasing on $(0,a)$ and decreasing on $(a,1)$.
For the first assertion, take
$0<s<t<a$ and write $t=\theta s+(1-\theta)a$. Then by concavity and affineness, 
\[
 \frac{F(t)}{L(t)}
 \ge\frac{\theta F(s)+(1-\theta)c}
          {\theta L(s)+(1-\theta)c}
 \ge\frac{F(s)}{L(s)} ,
\]
where the second inequality can be checked by subtraction. 
The decreasing assertion on $(a,1)$ follows by the same chord argument.
Therefore both
$F(as)/L(as)$ and $F(1-bs)/L(1-bs)$ are increasing with respect to $0<s<1$. 
Their weighted averages over $(t,1)$ with respect to weights $L(as)/s$ and $L(1-bs)/s$ are 
\[
 \frac{P_a(t)}{\overline P_a(t)},\qquad
 \frac{Q_a(t)}{\overline Q_a(t)} ,
\]
and hence those ratios are increasing in $t$. 
These ratios lie in $(0,1]$ since $0<F \leq L$, and clearly ${\overline P_a(t)}$ and ${\overline Q_a(t)}$ are decreasing.
We apply \eqref{eq:momentquotientorder} with $(Y,Z)=(P_a,{\overline P_a})$ and with $(Y,Z)=(Q_a,{\overline Q_a})$. Subsequently using
\eqref{eq:logrderivative} for $F$ and also for $L$ (which is valid since the derivation of that formula does not require zero
endpoint values) gives that 
\begin{equation}\label{eq:tangentlogcomparison}
 (\log R_F)'(a)\ge(\log R_L)'(a).
\end{equation}
In this formula the chosen line $L$ (which depends on the value $a$) is held fixed when differentiating
$R_L$. The supporting line is chosen separately at each $a$; we never differentiate that choice. 

The comparison in \eqref{eq:tangentlogcomparison} is actually strict, because a nonzero $F$ with zero endpoint values
cannot agree everywhere with an affine function. Thus $F<L$ on an
interval on at least one side of $a$, making at least one of the profile inequalities strict on a
set of positive measure. 

Next we show $(\log R_L)'(a) \geq 0$. For this remaining part of the argument, we ignore how $L$ was constructed in terms of $F$, and simply take $d_0$ and $d_1$ to be nonnegative constants that are not both zero and work with the affine function $L(t) = d_0(1-t) + d_1 t$. The parameter $a \in (0,1)$ is arbitrary. Put
\[
c = L(a) =d_0(1-a)+d_1 a 
\]
and define functions
\[
 h(t)=1-t,\qquad k(t)=-\log t-(1-t)\ge0.
\]
Direct integration gives
\[
 \overline P_a=ch+d_0 k,\qquad
 \overline Q_a=ch+d_1 k .
\]
For $c>0$ and $d\ge0$, set
\[
 N(c,d)=\left(\int_0^1(ch+dk)^q\,dt\right)^{\!1/q}.
\]
We now claim the key property, that 
\begin{equation}\label{eq:logsubmodular}
 \partial_c\partial_d\log N(c,d)\le0,
\end{equation}
which holds by direct computation followed by Cauchy--Schwarz, when $d>0$, and extends to $d=0$ by continuity. 
The differentiations are justified by the
endpoint bounds $h\sim1-t$ and $k\sim(1-t)^2/2$ near $t=1$,
and by the at-most logarithmic growth near $t=0$. 

For the line $L$, we have
$R_L(a)=[N(c,d_0)/N(c,d_1)]^q$ and $c'=d_1-d_0$.
By \eqref{eq:logsubmodular}, $\partial_c\log N(c,d)$ is
decreasing in $d$. Hence the two factors in the following
derivative have the same sign, giving
\begin{equation*}
 (\log R_L)'(a)
 =q(d_1-d_0)
 \bigl[\partial_c\log N(c,d_0)
       -\partial_c\log N(c,d_1)\bigr]\ge0.
\end{equation*}
The strictness of \eqref{eq:tangentlogcomparison} then gives $(\log R_F)'(a)>0$.

Now we justify the earlier $a$-differentiations and
integrations by parts for continuous $F$. On compact subintervals of
$0<a<1$, the displayed derivatives are uniformly bounded,
while
\[
 0<P_a(t),Q_a(t)\le\norm{F}_\infty\log(1/t).
\]
Every finite nonnegative power of this upper bound is integrable,
and the integration by parts boundary terms vanish. Dominated
convergence gives $P_a\to0$ in $L^q$ as $a\downarrow0$, while $Q_a$
converges to a nonzero function. Reflection gives the opposite behavior
as $a\uparrow1$. Hence $R_F(a)$ strictly increases from $0$ to
$+\infty$, and is continuous.

Finally take $q=p-1$. The derivative identity
\eqref{eq:Jpderivative}, which is valid for every $p>1$, becomes
\[
 J_p'(a)
 =\frac{p}{4^p}F(a)M_q(Q_a)\bigl(R_F(a)-1\bigr).
\]
The strictly increasing ratio $R_F(a)$ crosses $1$ exactly once, and so $J_p(a)$ strictly decreases
and then strictly increases. Continuity at $a=0,1$ follows from
the logarithmic bound above. This argument completes the proof of the lemma.
\end{proof}

\subsubsection*{Completing the proof for $1<p<\infty, n=2$}

Lemmas~\ref{lem:lpendpoint} and \ref{lem:largelp} supply the endpoint
comparison for $J_{2,p}$ when $1<p<\infty$, including a strictness
assertion. The reduction \eqref{eq:dimensionpbound}--\eqref{eq:comparisonendpoints} then completes the proof of the inequality in Theorem~\ref{thm:planarlp} for these $p$-values, when $n=2$.

\subsection{Dimension $n=1$}

The one-dimensional argument is shorter because the radial weights are absent from $V_a$ and $W_b$. In fact, when $n=1$, definitions \eqref{eq:Vanhigher} and \eqref{eq:Wbnhigher} give  
\[
V_a(t) = \int_t^a F(s) \, ds , \qquad W_b(t) = \int_t^b \widetilde{F}(s) \, ds ,
\]
for $t \in [0,a]$ and $t \in [0,b]$, respectively. In our normalized setting, we have $|\Omega|=\omega_1=2$, 
$B=(-1,1)$ and $v(x)=V_1(|x|), w(x)=W_1(|x|)$. 

We first sketch a direct one-dimensional proof of Talenti's
comparison on a nodal set, since the proof is simple and informative.
Let $E_\tau=\{u_+>\tau\}$ and $\mu(\tau)=|E_\tau|$. At a regular
positive level $\tau$, integration of $-u''=f$ over the component
intervals of $E_\tau$, followed by Cauchy--Schwarz and the
one-dimensional isoperimetric inequality $P(E_\tau)\ge2$, gives
\[
 4\le
 \left(\int_{E_\tau}f\dd\right)(-\mu'(\tau))
 \le 2F\left(\frac{\mu(\tau)}2\right)(-\mu'(\tau)),
\]
where the final inequality holds by the measurable set
Hardy--Littlewood inequality \eqref{eq:top} with $\omega_1=2$.

Put $t=\mu(\tau)/2$ and invert the distribution function.
For the decreasing profile $U_+$ of $u_+$ this yields
\[
 -U_+'(t)\le F(t),
 \qquad
 U_+(t)\le\int_t^\alpha F(s)\,ds=V_\alpha(t).
\]
The same argument applied to $u_-$ and $-f$ gives
$U_-(t)\le W_\beta(t)$.

Next, the endpoint inequality is particularly transparent in one dimension.

\begin{lemma}[One-dimensional endpoint maximum]
\label{lem:oneendpoint}
For $1\le p<\infty$, let $J_p=J_{1,p}$ be the functional
from \eqref{eq:Jnp}. Then
\[
 J_p(a)\le\max\{J_p(0),J_p(1)\}.
\]
If $F\not\equiv0$ then the inequality is strict for $0<a<1$.
\end{lemma}

\begin{proof}
If $F\equiv0$, then $J_p\equiv0$ and the conclusion is immediate.
Assume $F\not\equiv0$, so that $F(a)>0$ for $0<a<1$. For $n=1$, the definitions give
\[
 V_a(t) = \int_t^a F(s) \,ds , \qquad  W_{1-a}(t) = \int_t^{1-a} \widetilde{F}(s) \,ds ,
\]
for $t \in [0,a]$ and $t \in [0,1-a]$ respectively. Continuity of $J_p(a)$ follows easily.

For $p=1$, differentiation gives
\[
 J_1'(a)=aF(a)-(1-a)\widetilde F(1-a)
        =(2a-1)F(a).
\]
Hence $J_1$ strictly decreases on $(0,1/2)$ and strictly increases
on $(1/2,1)$. Every interior value is therefore smaller than the
larger endpoint value.

Now let $1<p<\infty$. We use the profiles $P_a,Q_a$ from \eqref{eq:rescaledprofiles} with $n=1$ and the moments $M_r$ from \eqref{eq:profilemoments}. The profile formulas become
\[
 \begin{aligned}
 P_a(t)&=V_a(at)=\int_{at}^aF(s)\,ds,\\
 Q_a(t)&=W_{1-a}((1-a)t)
       =\int_{(1-a)t}^{1-a}\widetilde F(s)\,ds.
 \end{aligned}
\]
Writing $D_p=D_{1,p}$ as in \eqref{eq:momentdifference}, we have
\[
 \begin{aligned}
 D_p(a)&=aM_{p-1}(P_a)-(1-a)M_{p-1}(Q_a)\\
       &=\int_0^aV_a(t)^{p-1}\,dt
        -\int_0^{1-a}W_{1-a}(t)^{p-1}\,dt.
 \end{aligned}
\]
The factors $a$ and $1-a$ come from the changes of variable in
these two integrals. Since $\partial_aV_a(t)=F(a)$ and
$\partial_bW_b(t)=\widetilde F(b)$, with
$\widetilde F(1-a)=F(a)$, differentiation gives
\begin{equation}\label{eq:oneJderivative}
 J_p'(a)=pF(a)D_p(a).
\end{equation}
The difference $D_p$ is strictly increasing, as the following
calculation shows:
\begin{equation}\label{eq:oneDderivative}
 D_p'(a)=(p-1)F(a)
 \left[aM_{p-2}(P_a)+(1-a)M_{p-2}(Q_a)\right]>0,
 \qquad 0<a<1.
\end{equation}
To justify the differentiations, note that the profiles vanish linearly at
their moving endpoints, with $V_a(a)=0, V_a^\prime(a)=-F(a)<0$ and similarly for $W_{1-a}$. Thus the $M_{p-2}$ moments in \eqref{eq:oneDderivative} are
finite. 

The function $D_p$ is continuous on $[0,1]$, since the profiles depend continuously on $a$ and the exponent $p-1$ is positive, and 
at the endpoints it takes values 
\[
 D_p(0)=-\int_0^1W_1(t)^{p-1}\,dt<0,
 \qquad
 D_p(1)=\int_0^1V_1(t)^{p-1}\,dt>0.
\]
Consequently $D_p$ has a unique zero $c\in(0,1)$. By
\eqref{eq:oneJderivative}, $J_p$ strictly decreases on $(0,c)$ and
strictly increases on $(c,1)$. Every interior value is therefore
smaller than the larger endpoint value.
\end{proof}

For $n=1$, Lemma~\ref{lem:oneendpoint} and the reduction
\eqref{eq:dimensionpbound}--\eqref{eq:comparisonendpoints} prove
Theorem~\ref{thm:planarlp} for every finite $p$, including
$p=1$. The case $p=\infty$ follows
from Theorem~\ref{thm:linfty}, as remarked earlier.

\subsection*{Equality in Theorem~\ref{thm:planarlp} for finite $p$}

Assume $1 \leq p < \infty$, $f\not\equiv0$, and equality holds in
\eqref{eq:lpquestion}. If $n=2$ and $p>1$, or if $n=1$, then strictness
of the endpoint inequality as proved above forces
$\alpha=0$ or $\alpha=1$. The first possibility forces $\beta=1$, since 
otherwise the strict increase of $W_b$ with $b$ would give
$\norm{u}_{L^p}<\norm{w}_{L^p}$. Equality in the nodal comparison
then gives $U_-=W_1$ or $U_+=V_1$, respectively.
Lemma~\ref{lem:nodalrigidity} proves the claimed domain and solution
identities.

Now let $(n,p)=(2,1)$. Equality in the calculation in
Section~\ref{sec:planarlone}, together with $F>0$ on $(0,1)$, forces
$\alpha+\beta=1$ and saturation of each nonempty nodal comparison.
Since $u$ is continuous in $\Omega$, Lemma~\ref{lem:nodalrigidity}
shows that each nonempty nodal set is a disk with a set of zero
$H^1$-capacity possibly removed. If only one nodal set is nonempty,
its volume fraction is $1$, and the same lemma gives the single-disk
domain and solution identities.

If both nodal sets are nonempty, let $D_+$ and $D_-$ be their
supporting disks. These disks are disjoint, and $\alpha+\beta=1$
implies that $\Omega$ agrees almost everywhere with their union.
Openness then forces $\Omega$ to be contained in the union. Any missing set has zero $H^1$-capacity, by the nodal set conclusion
of Lemma~\ref{lem:nodalrigidity}. 
On the two disks, $u$ has the translated profiles $V_\alpha$ and
$-W_\beta$. The identities
\[
 V_\alpha(t)=V_1(t)-V_1(\alpha),\qquad 0\le t\le\alpha,
 \qquad
 W_\beta(t)=W_1(t)-W_1(\beta),\qquad 0\le t\le\beta,
\]
show that these are precisely the truncated comparison potentials,
with their boundary values subtracted as stated in the theorem.

The $p=\infty$ equality case is proved in Section~\ref{sec:linfty}.

\subsection*{Counterexamples when $p<\infty$ to paired comparison with nonzero source mean}

The following family of examples shows why mean zero is needed for each 
$L^p$ comparison with $p<\infty$. The larger interval in the example carries mean-zero source,
while a small second interval changes the total mean. Sufficiently small values of the parameter give
counterexamples; the required smallness depends on $p$. 

\begin{proposition}[Paired comparison failure without mean zero]
\label{prop:finitep-arbitrarymean}
If $1\le p<\infty$ and $L>0$, then there exist a bounded open set
$\Omega\subset\R$ consisting of two intervals with $|\Omega|=L$
and a source $f$ on $\Omega$ taking only the values $1$ and $-2$, with nonzero
mean, such that
\[
 \norm{u}_{L^p(\Omega)}
 >\max\{\norm{v}_{L^p(B)},\norm{w}_{L^p(B)}\}
\]
where $B=(-L/2,L/2)$. 
\end{proposition}

\begin{proof}
Fix $p$ and rescale in order to normalize the length to $L=2$. For $0 \leq \delta<1$, let 
\[
 R=1-\delta,\qquad
 \Omega_\delta=(-R,R)\cup(3-\delta,3+\delta),\qquad B=(-1,1),
\]
and
\[
 f_\delta(x)=
 \begin{cases}
 1,&|x|<2R/3,\\
 -2,&\text{elsewhere in }\Omega_\delta.
 \end{cases}
\]
Then $\int_{\Omega_\delta}f_\delta(x)\,dx=-4\delta$, which is negative when $\delta>0$.

On the larger interval $(-R,R)$ of $\Omega_\delta$, the solution of $-u_\delta^{\prime\prime}=f_\delta$ is $R^2U(|x|/R)$, where
\[
 U(t)=
 \begin{cases}
 \displaystyle\frac13-\frac{t^2}{2},&0\le t\le2/3,\\[1mm]
 (1-t)^2,&2/3\le t\le1.
 \end{cases}
\]
On the smaller interval of $\Omega_\delta$ the solution is $u_\delta(x)=(x-3)^2-\delta^2$. Since $f^*$ equals $1$ on $|x|<2R/3$ and $-2$ on $2R/3<|x|<1$, direct calculation gives that 
\[
 \norm{u_\delta}_p^p-\norm{v_\delta}_p^p
 =2\delta^2 R^{2p-1} p \int_0^1 U(t)^{p-1} \,dt + o(\delta^2).
\]
Hence $\norm{u_\delta}_p>\norm{v_\delta}_p$
for all small $\delta>0$, noting that the required smallness depends on $p$. 

The comparison with $w$ is easier to determine, since one computes that $u_0>w_0 \geq 0$ pointwise except at $x=0$, and so $\norm{u_\delta}_p>\norm{w_\delta}_p$ for all small $\delta$. 
\end{proof}

\section{Remarks --- paired comparison method is nonsymmetric, and fails in three dimensions}
\label{sec:dimensionfailure}

Three features of the planar $L^2$ comparison deserve emphasis.
The $J$-functional can be nonsymmetric, and yet its endpoints still
dominate in dimension two. In dimension three, the functional can instead rise
above both endpoints, and furthermore we find an example for which Theorem~\ref{thm:planarlp} actually fails. The third subsection below identifies the dimensionally dependent weight that blocks the planar proof from extending to $n \geq 3$.

\subsection{Paired comparison can be nonsymmetric}
\label{sec:single-source-nonsymmetric}

The planar endpoint principle does not rely on 
symmetry of the $J$-functional. 

Consider disjoint disks with centers $c_+$ and $c_-$ sufficiently far apart, and set
\[
 \Omega=B_{\sqrt{2/3}}(c_+)\cup B_{\sqrt{1/3}}(c_-),
 \qquad
 f(x)=\begin{cases}
 3,&x\in B_{1/\sqrt3}(c_+),\\
 -3/2,&\text{elsewhere in }\Omega.
 \end{cases}
\]
The source has mean zero, $|\Omega|=\pi$ and $B=B_1$. Its integrated rearrangement
is the nonsymmetric piecewise linear (tent) function
\[
 F(t)=\min\{3t,\tfrac32(1-t)\}.
\]
Here $J=J_{2,2}[F]$.
The Dirichlet solution of $-\Delta u = f$ on $\Omega$ is
\[
 u(x)=\begin{cases}
 V_{2/3}(|x-c_+|^2),&x\in B_{\sqrt{2/3}}(c_+),\\
 -W_{1/3}(|x-c_-|^2),&x\in B_{\sqrt{1/3}}(c_-),
 \end{cases}
\]
which is positive on the larger disk and negative on the smaller one.
Consequently $\norm{u}_2^2=\pi J(2/3)$, while
$\norm{v}_2^2=\pi J(1)$ and $\norm{w}_2^2=\pi J(0)$.
With the help of \eqref{eq:tonelli} we calculate
\[
 J(2/3)=\frac{23}{384}-\frac{3}{64}\log2 \approx 0.0274,
 \qquad
 \norm{w}_2\approx0.280<\norm{u}_2\approx0.293
 <\norm{v}_2\approx0.316.
\]
Thus $w$ has smaller $L^2$ norm than $u$, which in turn has smaller norm than $v$. This ordering is shown in the left panel of
Figure~\ref{fig:endpointfunctionals}, along with the lack of symmetry of the graph of $J$.

A second example shows that the preferred endpoint can change with the
exponent. On the unit disk, take
\[
 F(t)=F_{1/4}(t)+\frac12F_{9/10}(t),\qquad
 F_\theta(t)=
 \min\left\{\frac{t}{\theta},\frac{1-t}{1-\theta}\right\},
\]
noting that $F$ is concave with $F(0)=F(1)=0$. For the corresponding radial solutions, direct calculation gives
\[
 \norm{v}_2\approx0.454,\qquad
 \norm{w}_2\approx0.443,
 \qquad
 \norm{v}_\infty\approx0.594,\qquad
 \norm{w}_\infty\approx0.607.
\]
Thus $v$ has the larger $L^2$ norm, whereas $w$ has the larger
$L^\infty$ norm and hence the larger $L^p$ norm for all large $p$.

\subsection{Paired comparison can fail in three dimensions}

The right panel of
Figure~\ref{fig:endpointfunctionals} exhibits a dramatic failure of the endpoint comparison result in dimension three. Take
$0<\varepsilon<1/2$ and consider the admissible integrated rearrangements
\begin{equation}\label{eq:Fepsilon}
 F_\varepsilon(t)=
 \min\left\{\frac{t}{\varepsilon},\,1,\,
                  \frac{1-t}{\varepsilon}\right\}.
\end{equation}
Each $ F_\varepsilon$ is continuous and concave, and vanishes at both endpoints. It is, up to a constant factor, the integrated rearrangement of the concentrated opposite sources in the three-dimensional construction 
 of Section~\ref{sec:highercounterexamples} below. 

For the concrete choice $\varepsilon=1/8$, corresponding to
$q=1/2$ in Section~\ref{sec:highercounterexamples}, direct calculation gives the graph of $J(a)=J_{3,2}[F_{1/8}](a)$ shown in the right panel of Figure~\ref{fig:endpointfunctionals}. The graph rises well above both endpoint values in its central portion, although it does dip slightly below them just inside the endpoints.

\begin{figure}[!htb]
 \centering
 \includegraphics[width=0.98\linewidth]{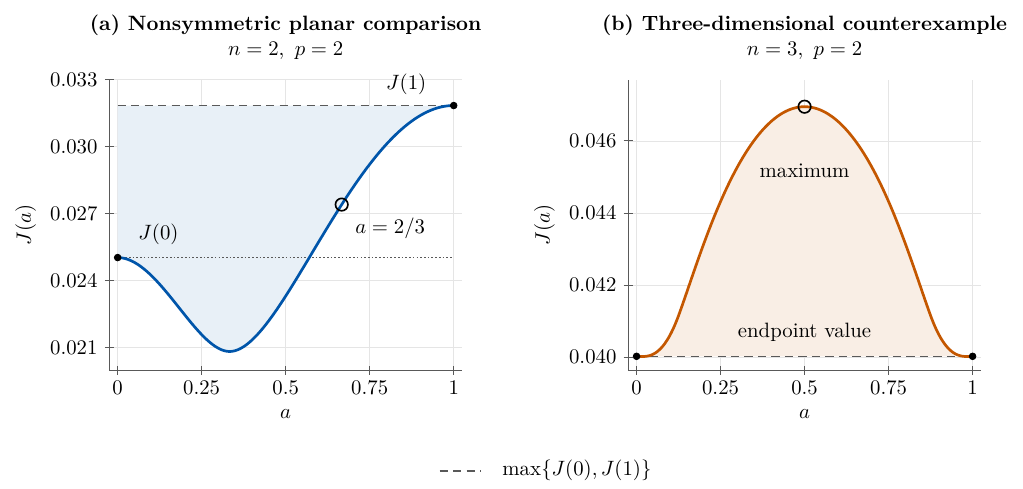}
 \caption{The $J$-functional in dimensions two and three:  $J(a)=J_{2,2}[F](a)$ for the piecewise linear integrated rearrangement $F(t)=\min\{3t,\tfrac32(1-t)\}$ in the left panel and
 $J(a)=J_{3,2}[F_{1/8}](a)$ in the right panel. The vertical scales differ. The dashed line is
 $\max\{J(0),J(1)\}$.
\emph{Left:} the planar example in
 Section~\ref{sec:single-source-nonsymmetric} produces a nonsymmetric
 graph $J(a)$. The marked value $J(2/3)=\norm{u}_2^2/\pi$ exceeds
 $J(0)=\norm{w}_2^2/\pi$ (the dotted line) but remains below
 $J(1)=\norm{v}_2^2/\pi$ (the dashed line). Thus the maximum in
 Theorem~\ref{thm:planarlp} cannot be replaced by the minimum.
 \emph{Right:} the three-dimensional integrated rearrangement
 $F_\varepsilon$ from \eqref{eq:Fepsilon} with
 $\varepsilon=1/8$ generates a symmetric graph $J(a)$ whose central
 values exceed both endpoint values; shallow dips occur near the
 endpoints. 
 }
 \label{fig:endpointfunctionals}
\end{figure}

\subsection{Why the proof of Theorem~\ref{thm:planarlp} breaks down in higher dimensions}

The signed Talenti comparison step remains valid in higher dimensions.
The step that fails is the endpoint inequality on the $J$-functional. As we shall show, the proof acquires a dimensionally dependent radial weight that is no longer controlled by concavity, when $n \geq 3$.

When $p=2$, differentiating the definition of $J_{n,2}$ and using the formulas \eqref{eq:Vanhigher} and \eqref{eq:Wbnhigher} for $V_a$ and $W_b$ gives that 
\begin{equation*}
 J_{n,2}'(a)=\frac{2F(a)}{n^4}\Bigg[
 a^{2/n-2}\int_0^a F(s)s^{2/n-1}\,ds 
 -(1-a)^{2/n-2}
 \int_a^1F(s)(1-s)^{2/n-1}\,ds 
 \Bigg] .
\end{equation*}
When $n=2$, the integrand weights disappear since the exponent $2/n-1$ equals $0$, and so the two integrals can be related through the total value $\int_0^1 F(s) \, ds$. But in dimensions $n \geq 3$, the integrands carry different singular weights and so cannot be related. The effect is that concavity of $F$ no longer forces
the sign pattern needed for the $J$-functional to have an endpoint maximum.

\section{Proof of Theorem~\ref{thm:linfty} --- \texorpdfstring{$L^\infty$}{L-infinity} comparison in all dimensions}
\label{sec:linfty}

\subsection*{One dimension with arbitrary mean}
\label{sec:linftyone}

In one dimension, an exact computation with the Green function permits us to treat arbitrary means, as follows.

Take $n=1$. By rescaling, we may take $|\Omega|=1$ and $B=(-1/2,1/2)$.
Use $m,f^\downarrow,F,\widetilde F$ from Section~\ref{sec:notation}.
Since $|\Omega|=1$, we have $m=\int_\Omega f\dd$ and
$f^*(x)=f^\downarrow(2|x|)$ for $|x|<1/2$.
Replacing $(f,u,v,w)$ by $(-f,-u,w,v)$ if necessary, we may suppose $m\ge0$. Set
\[
 A=\int_0^1 F(s)\,ds.
\]
The function $F$ is concave, with $F(0)=0$ and $F(1)=m$,
and hence $F(s)\ge ms\ge0$. The solution of $-v^{\prime\prime}=f^*$ is 
\[
 v(x)=\frac14\int_{2|x|}^1 F(s)\,ds,\qquad
 \norm v_\infty=v(0)=\frac A4.
\]
For the reflected source, recall that $\widetilde F(s)=F(1-s)-m$. 

Consider a component interval $I$ of length $\ell$. We might as well suppose $I=(0,\ell)$. On that interval the solution is 
\[
u(x) = \int_0^\ell G(x,y) f(y) \, dy , \qquad x \in (0,\ell) ,
\]
where the Dirichlet Green function is 
\[
G(x,y) 
= 
\begin{cases}
(1-x/\ell)y , & y \in (0,x] , \\
(1-y/\ell)x , & y \in [x,\ell) . 
\end{cases}
\]
Note that $G(x,y) \leq h(x) := (1-x/\ell)x$. By the layer cake representation, 
\begin{align*}
u(x) 
& = \int_0^{h(x)} \int_{\{ y : G(x,y) > z \}} f(y) \, dy dz \\
& \leq \int_0^{h(x)} F\big((1-z/h(x))\ell\big) \, dz 
\end{align*}
since the length of the superlevel set $\{ G(x,\cdot)>y \}$ depends linearly on $y$ and so equals $(1-y/h(x))\ell$. Changing variable and using that $h(x)/\ell \leq 1/4$, we find 
\[
u(x) \leq \frac{h(x)}{\ell} \int_0^\ell F(z) \, dz \leq \frac{1}{4} \int_0^\ell F(z) \, dz \leq \frac{1}{4} A
\]
since $F \geq 0$. Similarly, if $u(x) \leq 0$ then 
\[
0 \leq -u(x) \leq \frac{h(x)}{\ell} \int_0^\ell \widetilde F(z) \, dz \leq \frac{1}{4} \int_0^\ell \widetilde F(z) \, dz = \frac14\left(\int_{1-\ell}^1 F(s)\,ds-m\ell\right)
 \leq \frac A4 ,
\]
and the final inequality is strict when $m>0$. Consequently $\norm u_\infty\le\norm v_\infty$ for $m\ge0$. Further, applying this bound with $u$ replaced by $-w$, whose source \(-(-f)^*\) has nonnegative mean and decreasing rearrangement \(f^*\), gives \(\|w\|_\infty\le\|v\|_\infty\). Thus the $L^\infty$ norm of $v$ is greater than or equal to the norms of $u$ and $w$. 

When $m\le0$, the initial sign reversal gives the corresponding bounds with the roles of $v$ and $w$ interchanged. 

For the equality conclusion, suppose the common norm is positive.
The continuous zero extension of $u\in H^1_0(\Omega)$ has compact
support, and so $|u|$ attains its maximum at a point of $\Omega$.
Also $F>0$ on $(0,1)$, since $f\not\equiv0$ (recall $u$ has positive norm, in the equality statement) and $m\ge0$.
If $m>0$ then the argument above shows that the maximizing value is positive, and if $m=0$ then we may suppose the maximum is positive by replacing $u$ with $-u$ if necessary.
Equality in the bound $\norm u_\infty\le\norm v_\infty$ now forces $\ell=1$ and
$h(x) = \ell/4$. Thus $\Omega$ is one interval and $x$ is its midpoint.
At this point, equality in the rearrangement estimate forces
$f=f^*$ after translation: every centered Green superlevel interval
must contain source values of maximal integral, and the Green
function is strictly decreasing in distance from the center point. Hence $u=v$. Undoing any sign reversal gives $u=v$ or $u=-w$,
as required.

\subsection*{Mean-zero sources in all dimensions}

Now assume $\int_\Omega f\dd=0$ and $n \geq 1$. The positive and negative peaks will be
controlled separately, with Talenti nodal comparison
giving the bound in every dimension.

With notation from Section~\ref{sec:notation}, we have
\[
 \norm v_\infty=V_1(0),\qquad
 \norm w_\infty=W_1(0),
\]
where these limiting values at $t=0$ may be infinite. Since $U_+ \le V_\alpha$ and $U_- \le W_\beta$ by Lemma \ref{lem:level}, and $V_a$ and $W_b$ are increasing with respect to $a$ and $b$ respectively, we find 
\[
 \norm{u_+}_\infty\le V_\alpha(0)
 \le V_1(0)=\norm v_\infty
\]
and similarly
\[
 \norm{u_-}_\infty\le W_\beta(0)
 \le W_1(0)=\norm w_\infty.
\]
Taking the larger of the two
inequalities proves the $L^\infty$ bound in Theorem~\ref{thm:linfty}.

Suppose next that equality holds with finite positive common value
$M$ for $\norm{u}_{L^\infty(\Omega)}$ and the other norm. Then $f\not\equiv0$ and both comparison potentials are bounded.
After reversing the signs of $u$ and $f$ if needed, we may suppose
$\norm{u_+}_\infty=M$. Since $F>0$ on $(0,1)$, the preceding bounds
force $\alpha=1$ and $U_+(0)=V_1(0)=M$.

To pass from equality of the heights to equality of the profiles,
adapt Kesavan's distribution argument
\cite[Proposition~3.2.1, Step~2, pp.~53--54]{Kesavan2006}
to the signed source by using the Hardy--Littlewood bound
\eqref{eq:top} in the positive level estimate. This gives the
differential comparison
\[
 0\le-U_+'(t)\le-V_1'(t),\qquad 0<t<1,
\]
almost everywhere. The source enters this argument only through its positive
integrated rearrangement $F$, and pointwise nonnegativity of $f$ is not required.

The rearrangement $(u_+)^*$ belongs to $H^1_0(B)$, and so its radial
profile $U_+$ is absolutely continuous on $[\varepsilon,1]$
for every $0<\varepsilon<1$, with $U_+(1)=0$.
Integrating and letting $\varepsilon\downarrow0$, we obtain
\[
 \int_0^1(-U_+'(t))\,dt
 =U_+(0)=M=V_1(0)
 =\int_0^1(-V_1'(t))\,dt.
\]
The differential comparison and equality of these finite
integrals force $U_+'=V_1'$ almost everywhere.
The zero terminal values then give $U_+=V_1$.
Lemma~\ref{lem:nodalrigidity} now gives the equality conclusion
in Theorem~\ref{thm:linfty}.

Finally, suppose $f\in L^\ell(\Omega)$ for some finite
$\ell\ge1$ with $\ell>n/2$. H\"older's inequality and the zero
total integral give
\[
 |F(s)|\le\norm{f^\downarrow}_{L^\ell(0,1)}
 \min\{s^{1-1/\ell},(1-s)^{1-1/\ell}\}.
\]
Near zero, the integrands defining $V_1(0)$ and
$W_1(0)$ are bounded by a constant times
$s^{2/n-1-1/\ell}$, which is integrable because $\ell>n/2$.
Thus $v,w$ are bounded, and the comparison also bounds $u$.

\subsection*{Proof of $2^{1/p}$ comparison as remarked after Theorem~\ref{thm:linfty}}

Fix $1\le p<\infty$ and assume $\int_\Omega f\dd=0$.
Normalize $|\Omega|=\omega_n$, as in
Section~\ref{sec:notation}. The nodal comparison result (Lemma \ref{lem:level})
gives $U_+(t)\le V_\alpha(t)$ and $U_-(t)\le W_\beta(t)$.
Because the integrated rearrangement $F$ is nonnegative, the
integral formulas for the truncated potentials give
\[
 V_\alpha(t)\le V_1(t),\quad 0\le t\le\alpha,
 \qquad
 W_\beta(t)\le W_1(t),\quad 0\le t\le\beta.
\]
Thus, with the norm integrals understood in the extended sense, possibly equalling $+\infty$, we have 
\begin{align*}
 \int_\Omega |u|^p\dd
 &=\omega_n\int_0^\alpha U_+(t)^p\,dt
   +\omega_n\int_0^\beta U_-(t)^p\,dt\\
 &\le\omega_n\int_0^\alpha V_1(t)^p\,dt
      +\omega_n\int_0^\beta W_1(t)^p\,dt\\
 &\le\int_Bv^p\dd+\int_Bw^p\dd.
\end{align*}
Taking $p$-th roots gives the estimate \eqref{eq:2p} after Theorem~\ref{thm:linfty}.

\section{Proof of Theorem~\ref{thm:bangbang} --- bang-bang comparison}
\label{sec:bangbangproof}

When the source $f$ is bang-bang, the function $J_{n,p}(a)$ becomes helpfully explicit. 
We show that its maximum occurs at an endpoint of $[0,1]$ and then identify the larger endpoint.

Fix $n \geq 3$ and suppose $p$ is finite, with $(3n-1)/4\le p<\infty$. Normalize
$|\Omega|=\omega_n$ and define $\theta$ to be the volume fraction of the set $E$, so that
\[
 \theta=\frac{|E|}{\omega_n}
        =\frac{B_0}{A_0+B_0}\in(0,1),
\]
where we evaluated $\theta$ using the assumption that $f$ has mean value zero. 
Dividing the source and its potentials by $A_0\theta=B_0(1-\theta)$,
we may assume without loss of generality that the integrated rearrangement is
\[
 F_\theta(t)=
 \min\left\{\frac{t}{\theta},\frac{1-t}{1-\theta}\right\}.
\]
For fixed $\theta$, write $F=F_\theta$ and $J_p=J_{n,p}[F_\theta]$ as in \eqref{eq:Jnp}. The reduction
\eqref{eq:dimensionpbound}--\eqref{eq:comparisonendpoints}
shows that it suffices to prove
\begin{equation}\label{eq:bangbangendpoint}
J_p(a)\le\max\{J_p(0),J_p(1)\},
 \qquad 0\le a\le1.
\end{equation}
We will show the inequality is strict when $0<a<1$. 

Put $b=1-a$ and recall the rescaled profiles from \eqref{eq:rescaledprofiles}:
\[
 P_a(t)=n^2V_a(at),\qquad
 Q_a(t)=n^2W_b(bt),\qquad 0<t<1.
\]
Using the moment notation $M_r$ from \eqref{eq:profilemoments}, we have
\[
 J_p(a)=\frac1{n^{2p}}\left[aM_p(P_a)+bM_p(Q_a)\right].
\]
The linear growth bounds $F(t)\le t/\theta$ and
$\widetilde F(t)\le t/(1-\theta)$ make $V_a$ and $W_b$ 
bounded and give continuity of $J_p(a)$ on $[0,1]$.
For interior $a$, the linear vanishing of $V_a$ and $W_b$ at the moving endpoints,
together with $p\ge(3n-1)/4\ge2$, justifies the moment differentiations below.
Write $D_p=D_{n,p}$ for the moment difference in \eqref{eq:momentdifference}.
Differentiation like in the proof of Lemma \ref{lem:largelp} gives
\[
 J_p'(a)=\frac{pF(a)}{n^{2p}}D_p(a).
\]
Reflection replaces $(\theta,a)$ by $(1-\theta,1-a)$ and exchanges
the roles of $V_a$ and $W_b$. It therefore suffices to analyze 
\[
\theta\le a\le1
\]
for arbitrary $\theta$.
In particular, this choice implies $b\le1-\theta$, leading to the following exact expression when $0\le t\le b$: 
\begin{equation}\label{eq:bangbangtorsion}
 W_b(t)=\frac{b^{2/n}-t^{2/n}}{2n(1-\theta)}.
\end{equation}
Also, $V_a$ is convex and decreasing
as a function of $t^{2/n}$, since  
\[
 \frac{dV_a}{d(t^{2/n})}=-\frac{F(t)}{2nt}
\]
and $F(t)/t$ is decreasing by concavity of $F$.

\smallskip
\noindent{\bf Case 1: $\theta\le1/2$.} We will show $D_p(a) \geq 0$, so that $J_p(a)$ is increasing on $[\theta,1]$.
For $\theta\le a<1$, with $a$ fixed, normalize $P_a$ further by setting
\[
 X_{a,\theta}(x)=\frac{2(1-\theta)}{n b^{2/n}}P_a(x)
 =\frac{2(1-\theta)}{n b^{2/n}}
   \int_{ax}^a F_\theta(t)t^{2/n-2}\,dt.
\]
In the following comparison, $a$ is held fixed while $\theta$ varies. 
Then $D_p(a)\ge0$ is equivalent to
\begin{equation}\label{eq:bangbangmomentcomparison}
 \frac{\int_0^1X_{a,\theta}(x)^{p-1}\,dx}
      {\int_0^1(1-x^{2/n})^{p-1}\,dx}
 \ge\left(\frac ab\right)^{\!1-2/n}.
\end{equation}
Since $X_{a,\theta}(1)=0$ and the mapping
$x^{2/n}\mapsto X_{a,\theta}(x)$ is convex, it follows that
$x\mapsto X_{a,\theta}(x)/(1-x^{2/n})$ is decreasing, with
terminal value $(b/a)^{1-2/n}$. If $\theta\le a\le1/2$, this value is at
least $ 1$ implying \eqref{eq:bangbangmomentcomparison}. Consequently, we consider $1/2\le a<1$, so that $b/a \leq 1$. A direct comparison gives $X_{a,\theta}\ge X_{a,1/2}$, when $a,b$
are held fixed, since $(1-\theta)/\theta \geq 1$. Put  
\[
 \eta=
 \frac{2\left((1+b/a)/2\right)^{\!2/n}
             -1-(2/n)b/a}{1-2/n}\le 1.
\]
One computes that 
\[
 X_{a,1/2}(x)=(a/b)^{2/n}\bigl(\eta-x^{2/n}\bigr),
 \qquad 0\le x\le\frac1{2a}.
\]
As a function of $x^{2/n}$, $X_{a,1/2}$ is convex and agrees with
the displayed affine function when $ax\le1/2$. Consequently it
lies above that line for $ax>1/2$, and thus so must $X_{a,\theta}$, so that 
\[
 X_{a,\theta}(x)\ge
 \left(\frac ab\right)^{\!2/n}(\eta-x^{2/n})_+,
 \qquad 0\le x\le1.
\]
Hence
\[
\begin{aligned}
\int_0^1X_{a,\theta}(x)^{p-1}\,dx
 &\ge (a/b)^{2(p-1)/n}\eta^{p-1+n/2}
       \int_0^1(1-x^{2/n})^{p-1}\,dx.
\end{aligned}
\]
The elementary inequality $\eta\ge(b/a)^{2/(5n)}$ gives
\begin{equation*}
 \begin{aligned}
 \frac{\int_0^1X_{a,\theta}(x)^{p-1}\,dx}
      {\int_0^1(1-x^{2/n})^{p-1}\,dx}
 &\ge (a/b)^{2(p-1)/n}\eta^{p-1+n/2}\\
 &\ge (a/b)^{[8(p-1)/n-1]/5}\\
 &\ge\left(\frac ab\right)^{\!1-2/n},
 \end{aligned}
\end{equation*}
where the last step uses $a/b\ge1$ and $p\ge(3n-1)/4$.
Hence \eqref{eq:bangbangmomentcomparison} holds and so $D_p(a) \geq 0$.

In fact $D_p(a)>0$ for $a \in (\theta,1)$, as the following three cases show.
If $a<1/2$, the terminal value $(b/a)^{1-2/n}$ is greater than $1$,
whereas the right side of \eqref{eq:bangbangmomentcomparison} is less
than $1$. If $a>1/2$, the affine continuation for $X_{a,1/2}$ is
strictly below $X_{a,1/2}$ on an interval of positive length.
Since $X_{a,1/2}>0$ for $x<1$, replacing that affine continuation
by its positive part still gives a strict integral inequality. 
When $\theta<a=1/2$, the comparison
$X_{a,\theta}\ge X_{a,1/2}$ is strict on the interval $(0,1)$.
Consequently $J_p(a)<J_p(1)$ for $\theta\le a<1$ in Case~1.

\smallskip
\noindent{\bf Case 2: $\theta\ge1/2$.}
Still with $\theta \leq a < 1$, we prove that every critical point of $J_p(a)$ in $(\theta,1)$ is a
strict local minimum.
Put $\nu(y)=V_a(a y^{n/2})$ for $0 \leq y \leq 1$. Convexity with respect to $y$ implies 
$-y\nu'(y)\le \nu(0)-\nu(y)$. Since $\nu(1)=0$, integration by parts
gives
\[
 \frac n2\int_0^1y^{n/2-1}\nu^{p-1}\,dy
 =(p-1)\int_0^1y^{n/2}\nu^{p-2}(-\nu')\,dy
 \le(p-1)\int_0^1y^{n/2-1}\nu^{p-2}(\nu(0)-\nu)\,dy.
\]
Rearranging, and using $F(t)\le t/\theta$ to bound $\nu(0)$, yields
\begin{equation}\label{eq:bangbangmoment}
 \frac{M_{p-1}(P_a)}{M_{p-2}(P_a)}
 \le\frac{n^2(p-1)\nu(0)}{p-1+n/2},
 \qquad
\nu(0)\le\frac{a^{2/n}}{2n\theta}.
\end{equation}

At an interior critical point of $J_p$, we have $D_p(a)=0$, so that
 $
 a^{2/n-1}M_{p-1}(P_a)=b^{2/n-1}M_{p-1}(Q_a)
$.
Differentiating the moment difference in \eqref{eq:momentdifference} and using the explicit formula \eqref{eq:bangbangtorsion} gives
\[
 \frac{D_p'(a)}{a^{2/n-1}M_{p-1}(P_a)}
 =\frac{(p-1)b a^{2/n-2}}{1-\theta}
      \frac{M_{p-2}(P_a)}{M_{p-1}(P_a)}
  +\frac{2/n-2}{a}+\frac{2p/n-1}{b}.
\]
Since $\theta\ge1/2$ and $b=1-a$, we have
\[
 a+\frac{\theta b^2}{(1-\theta)a}
 \ge a+\frac{b^2}{a}
 =2a-2+\frac1a\ge2\sqrt2-2.
\]
Also, $p\ge(3n-1)/4$ gives
$1+2(p-1)/n\ge(5/2)(1-1/n)$.
Consequently, \eqref{eq:bangbangmoment} yields
\begin{equation*}
 \begin{aligned}
 \frac{abD_p'(a)}{a^{2/n-1}M_{p-1}(P_a)}
 &\ge\left(1+\frac{2(p-1)}n\right)
       \left(a+\frac{\theta b^2}{(1-\theta)a}\right)
       +\frac2n-2\\
 &\ge\left(1-\frac1n\right)(5\sqrt2-7)
       \ge\frac{10\sqrt2-14}{3}>0.
 \end{aligned}
\end{equation*}
The last bound uses $n\ge3$.
Thus every critical point in $(\theta,1)$ is a strict local minimum, and so the maximum of $J_p(a)$ on $[\theta,1]$ occurs at $\theta$ or $1$.

\smallskip
\noindent{\bf Conclusion.}
If $\theta\le1/2$ then Case 1 gives $J_p(\theta) < J_p(1)$,
while the reflection of Case 2 shows that the maximum of $J_p$ on $[0,\theta]$ is attained only at an endpoint $0$ or $\theta$. Hence the maximum on $[0,1]$ occurs only at $0$ or $1$.
If $\theta\ge1/2$ then the reflection of Case 1 shows that $J_p(\theta) < J_p(0)$, while 
Case 2 shows that the maximum on $[\theta,1]$ occurs only at an endpoint $\theta$ or $1$, so that again the maximum on $[0,1]$
occurs only at $0$ or $1$. 
{
Together, these conclusions prove \eqref{eq:bangbangendpoint}, with strict inequality holding when $0<a<1$.
}
 
\subsubsection*{Which endpoint gives the larger norm?} 
If $\theta=1/2$, reflection gives $V_1=W_1$ and so the endpoints give the same value. In this case $v=w$. 

Suppose $\theta<1/2$. We establish an ordering by the following
crossing argument. Put
\[
 H(t)=F(t)-F(1-t),\qquad Z(t)=V_1(t)-W_1(t).
\]
The function $H$ is positive on $(0,1/2)$ and antisymmetric about
$1/2$, while
\[
 Z'(t)=-\frac1{n^2}H(t)t^{2/n-2},\qquad Z(1)=0.
\]
Thus $Z(t)<0$ for $1/2\le t<1$. Pairing $t$ with $1-t$ gives
\[
 Z(0)=\frac1{n^2}\int_0^{1/2}H(t)
      \left(t^{2/n-2}-(1-t)^{2/n-2}\right) dt>0.
\]
Since $Z$ is strictly decreasing on $(0,1/2)$, it crosses zero
exactly once, at some $t_0<1/2$. Next, a second paired integration gives
\[
 \int_0^1Z(t)\,dt
 =\frac1{n^2}\int_0^{1/2}H(t)
      \left(t^{2/n-1}-(1-t)^{2/n-1}\right) dt>0.
\]
Let $c=V_1(t_0)=W_1(t_0)>0$. Both profiles are decreasing.
Before the crossing, both profiles lie above $c$; after the crossing,
both lie below $c$ and their difference is negative.
The integral form of the derivative of $z \mapsto z^p$ therefore yields
\[
 V_1(t)^p-W_1(t)^p\ge pc^{p-1}(V_1(t)-W_1(t)) = pc^{p-1}Z(t)
\]
on both sides of the crossing. Integration therefore proves
\begin{equation*}
 \int_0^1V_1(t)^p\,dt - \int_0^1W_1(t)^p\,dt > 0,
\end{equation*}
so that $v$ has strictly larger $L^p$ norm than $w$. For $\theta>1/2$, reflection reverses the ordering.

Lastly, suppose $p=\infty$. Theorem~\ref{thm:linfty} gives the norm inequality. For ordering, argue as above: the argument when $\theta<1/2$ says $V_1(0)-W_1(0)=Z(0)>0$, so that $v$ has larger $L^\infty$ norm than $w$. 

\subsection*{Equality in Theorem~\ref{thm:bangbang}}

Suppose equality holds in the theorem. For $(3n-1)/4 \leq p < \infty$, the strict endpoint bound \eqref{eq:bangbangendpoint} for $0<a<1$ forces one nodal fraction to be $1$. Equality in the norm comparison then saturates the
corresponding nodal profile, and Lemma~\ref{lem:nodalrigidity} gives
the ball and solution identities. For $p=\infty$, the source is
bounded and nonzero, and so the equality statement of
Theorem~\ref{thm:linfty} applies. The strict endpoint ordering proved above for
$\theta\ne1/2$ selects $u=v$ when $\theta<1/2$ and $u=-w$ when
$\theta>1/2$; for $\theta=1/2$ the comparison potentials coincide.

Conversely, the stated domain and solution identities give equality in the theorem, since the selected comparison potential
has the larger norm by the endpoint ordering proved above.

\section{\texorpdfstring{$L^p$}{Lp} counterexamples in higher dimensions}
\label{sec:highercounterexamples}

The following construction produces a family of 
counterexamples to the paired comparison in every dimension $n\ge3$ and for every finite
exponent $1\le p<\infty$. The examples depend on $p$. Each fixed member of the family eventually
satisfies the comparison as $p$ grows to infinity, and so there is no contradiction with the $L^\infty$ theorem.

Even a single ball can provide a counterexample, as we demonstrate at the end of the section for $n=3$ dimensions with $p=2$. 

\subsection*{Constructing the \texorpdfstring{$L^p$}{Lp} counterexamples}

Fix $n\ge3$ and $1\le p<\infty$. Put $R=2^{-1/n}$, choose
$c_\pm=\pm2e_1$, and define
\[
 \Omega=B_R(c_+)\cup B_R(c_-),\qquad B=B_1,
\]
so that $|\Omega|=|B|$. For $0<q\le R$, define a source $f_q$ on $\Omega$ by
\[
 f_q=\mathbf1_{B_q(c_+)}-\mathbf1_{B_q(c_-)}.
\]
This source has integral zero. The rearrangements of $\pm f_q$ coincide
and have
\[
 f_q^\ast(r)=(-f_q)^\ast(r)=
 \begin{cases}
 1,&0\le r<q,\\
 0,&q<r<r_q,\\
 -1,&r_q<r<1,
 \end{cases}
\]
where $r_q=(1-q^n)^{1/n}$. Let $u$ be the potential of $f_q$
on $\Omega$, and let $v=w$ be the potential of $f_q^\ast=(-f_q)^\ast$ on $B_1$, with Dirichlet boundary conditions. Figure~\ref{fig:concentratedprofiles} displays the source and
potentials for one representative member of this family.

\begin{figure}[htbp]
 \centering
 \includegraphics[width=0.98\linewidth]{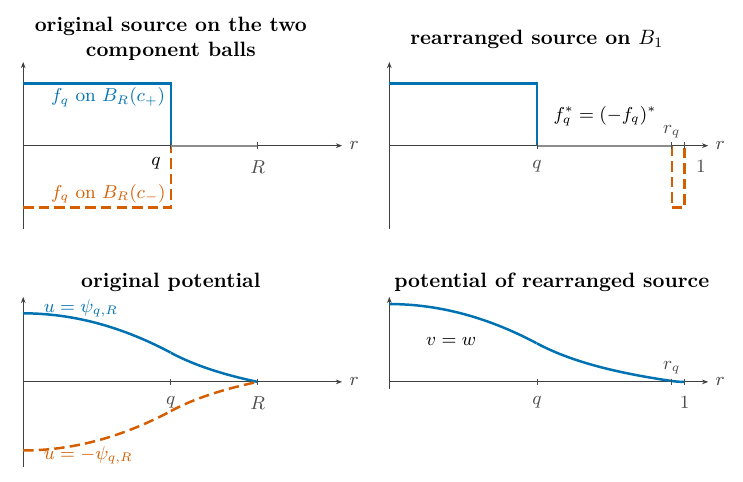}
 \caption{Source and potentials for the construction in
 Section~\ref{sec:highercounterexamples}, with $n=3$ and $q=1/2$.
 Here $R=2^{-1/3}$. The source $f_q$ is a positive central indicator
 on $B_R(c_+)$ and the opposite indicator on $B_R(c_-)$; correspondingly,
 $u$ consists of two opposite copies of $\psi_{q,R}$. The horizontal
 axes use the radial variable $r$; the original potentials end at
 $r=R$, while the potential for the rearranged source ends at $r=1$.}
 \label{fig:concentratedprofiles}
\end{figure}

For $s\ge q$, let $\psi_{q,s}$ solve
\[
 -\Delta\psi_{q,s}=\mathbf1_{B_q}
 \quad\text{in }B_s,\qquad
 \psi_{q,s}=0\quad\text{on }\partial B_s .
\]
Explicitly, 
\[
 \psi_{q,s}(r)
 =q^2 \left[ \Psi \! \left(\frac rq\right)
 -\Psi \! \left(\frac sq\right) \right],
 \qquad 0\le r\le s,
\]
where
\[
 \Psi(\rho)=
 \begin{cases}
 \displaystyle\frac1{2(n-2)}-\frac{\rho^2}{2n},
     &0\le\rho\le1,\\[5pt]
 \displaystyle\frac{\rho^{2-n}}{n(n-2)},
     &\rho\ge1,
 \end{cases}
\]
satisfies $-\Delta \Psi = \mathbf1_{B_1}$ in $\R^n$ and $\Psi(\rho) \to 0$ as $\rho \to \infty$. For our fixed exponent $p$, set
\[
 I_s(q)=\int_{B_s}\psi_{q,s}^{\,p}\dd.
\]
The solution $u$ is a translate of $\psi_{q,R}$ on one component and
its negative on the other, and so
\[
 \int_\Omega |u|^p\dd=2I_R(q).
\]
The nonnegativity of $v$ follows from Lemma~\ref{lem:comparison},
while the maximum principle gives $v\le\psi_{q,1}$. Hence
\[
 \int_{B_1}v^p\dd\le I_1(q).
\]
Thus to obtain a counterexample to the paired comparison inequality it
suffices to prove for some $q$ that $2I_R(q)>I_1(q)$, because then the $L^p$ norms of $v$ and $w$ are smaller than the norm of $u$. 

Integration of $\psi_{q,s}$ gives
the following three regimes as $q \downarrow 0$:
\[
 I_s(q)\sim
 \begin{cases}
 Cs^{\,n-(n-2)p}q^{np},&1\le p<p_c,\\
 Cq^{np}\log(1/q),&p=p_c,\\
 Cq^{n+2p},&p>p_c,
 \end{cases}
\]
where $p_c=n/(n-2)$ and the positive constants are independent of $s$ but depend on $n$ and $p$. Consequently,
\[
 \lim_{q\downarrow0}\frac{2I_R(q)}{I_1(q)}
 =
 \begin{cases}
 2^{p(n-2)/n},&1\le p<p_c,\\
 2,&p\ge p_c.
 \end{cases}
\]
In each case the limit is greater than $1$. 
Hence for every sufficiently small $q>0$,
\[
 \int_\Omega |u|^p\dd
 =2I_R(q)>I_1(q)
 \ge\int_{B_1}v^p\dd ,
\]
so that the paired comparison fails.

\subsection*{Why these counterexamples do not extend to $p=\infty$}

The construction fixes $p$ before choosing $q$. For each fixed $q>0$,
the solutions are bounded and $u$ changes sign, and so the equality case
of Theorem~\ref{thm:linfty} gives strict comparison at $p=\infty$.
Convergence of the $L^p$ norms then gives comparison for all large $p$.

\subsection*{Counterexample on a ball}

The paired comparison inequality can fail even when the domain is a
ball, as the following example shows for $L^2$ in three dimensions: two opposite concentrated
sources can give a larger $L^2$ norm than the solution of the rearranged problem.

Let $\Omega=B_1$ be the unit ball in three dimensions. For $0<q<1/2$, put
\[
 f_q=\mathbf1_{B_q(e_1/2)}
       -\mathbf1_{B_q(-e_1/2)}.
\]
The source has integral zero. Let $u_q$ be its Dirichlet
potential, and let $v_q,w_q$ be the Dirichlet potentials of
$f_q^*$ and $(-f_q)^*$ on $B_1$. The two rearranged sources
coincide, and so $v_q=w_q$.

Write $G(x,y)$ for the Dirichlet Green function of $B_1$, and set
\[
 H(x)=G(x,e_1/2)-G(x,-e_1/2).
\]
Then
\[
 \frac{u_q}{|B_q|}\longrightarrow H
 \quad\text{in }L^2(B_1)
 \qquad \text{as } q\downarrow0 .
\]

The rearranged source equals $1$ on $B_q$, equals $-1$ on the
outer shell of volume $|B_q|$, and vanishes elsewhere.
Its cumulative radial integral is nonnegative, and so the
radial solution formula and the maximum principle give
\[
 0\le v_q=w_q\le\psi_{q,1} 
\]
where $\psi_{q,1}$ is the Dirichlet potential of
$\mathbf1_{B_q}$ in the unit ball. Observe that
\[
 \frac{\psi_{q,1}}{|B_q|}\longrightarrow G(\,\cdot\,,0)
 \quad\text{in }L^2(B_1).
\]

Paired comparison fails for all sufficiently small $q$, since $H^2$ has larger integral than $G(\,\cdot\,,0)^2$:
\[
\int_{B_1}H(x)^2\,dx \approx 0.035 > 0.027 \approx \int_{B_1}G(x,0)^2\,dx .
\]

\section{Clamped counterexamples, and \texorpdfstring{$L^2$}{L2} conjecture in three dimensions}
\label{sec:clampedcounterexamples}

For a smooth bounded open set and
$u\in H^2(\Omega)\cap H^1_0(\Omega)$, the mean-zero assumption on
$f=-\Delta u$ has the boundary flux interpretation
\[
 \int_{\partial\Omega}\partial_\nu u\,dS
 =\int_\Omega\Delta u\dd=-\int_\Omega f\dd=0.
\]
Thus, under these regularity assumptions, the normal derivative has
mean zero on the boundary. Requiring that derivative to vanish
pointwise (the clamped boundary condition) still does not rescue the
comparison in dimensions $n\ge4$, as the following construction shows.
We work on smooth balls and arrange that all three
solutions satisfy both zero Dirichlet and zero normal derivative
conditions. The limiting ratio below captures the fact that the correction needed
to clamp the boundary remains bounded while the concentrated source
potential grows in $L^2$.

\subsection*{Clamped solutions and concentration limit}

For $n\ge2$, keep $R=2^{-1/n}$ and
$\Omega=B_R(c_+)\cup B_R(c_-)$ from
Section~\ref{sec:highercounterexamples}. Thus $B=B_1$. For
$0<\epsilon<1/4$, set $q=\epsilon^{1/n}$ and define
\begin{equation}\label{eq:clampeddata}
 f_\epsilon(x)=
 \begin{cases}
 \displaystyle\epsilon^{-1}\mathbf1_{B_q(c_+)}(x)-2,
      &x\in B_R(c_+),\\[3pt]
 \displaystyle-\epsilon^{-1}\mathbf1_{B_q(c_-)}(x)+2,
      &x\in B_R(c_-).
 \end{cases}
\end{equation}
The factor $\epsilon^{-1}$ makes the concentration explicit: the
small central ball has volume $\omega_n\epsilon$, while the
concentrated term carries total mass $\omega_n$. The integral of
$f_\epsilon$ on each component is zero. Let $u_\epsilon,v_\epsilon,
w_\epsilon$ denote the Dirichlet solutions for $f_\epsilon$,
$f_\epsilon^\ast$, and $(-f_\epsilon)^\ast$, respectively.

The sources in \eqref{eq:clampeddata} are radial on their
respective balls and have integral zero on each component. The radial flux
identity therefore gives
\[
 \partial_\nu u_\epsilon=0\qquad\text{on }\partial\Omega.
\]
Together with the Dirichlet condition and elliptic regularity on the
balls, this argument shows that $u_\epsilon\in H^2_0(\Omega)$. The two signed
source distributions agree, and so $v_\epsilon=w_\epsilon$. Their common
radial source also has integral zero; the same flux identity at the
boundary of $B_1$ gives
\[
 \partial_\nu v_\epsilon=0\qquad\text{on }\partial B_1.
\]
Hence $v_\epsilon\in H^2_0(B_1)$ as well.

The original solution $u_\epsilon$ has two concentrated peaks,
whereas the rearranged solution $v_\epsilon$ has only one. For $n\ge4$,
the concentration estimates in Section~\ref{sec:highercounterexamples} give directly
\begin{equation*}
 \lim_{\epsilon\downarrow0}
 \frac{\int_\Omega u_\epsilon^2\dd}
 {\int_{B_1}v_\epsilon^2\dd}=2,
 \qquad n\ge4.
\end{equation*}
Thus every $n\ge4$ admits a strict counterexample to the clamped
$L^2$ paired comparison for all sufficiently small positive
$\epsilon$. All three solutions are clamped, and the fixed open set
$\Omega$ has smooth boundary and two components.

\subsection*{Three-dimensional clamped comparison conjecture}

The concentration argument above produces clamped counterexamples
in dimensions $n\ge4$ but leaves dimension three unresolved.

\begin{conjecture}[Clamped $L^2$ comparison in dimension three]
Let $\Omega\subset\R^3$ be a nonempty bounded open set, and let
$B=B_R(0)$ be the ball of the same volume.
For any real-valued
$u\in H^2_0(\Omega)$, set $f=-\Delta u\in L^2(\Omega)$.
Let $v,w\in H^1_0(B)$ be the weak Dirichlet solutions of
\[
 -\Delta v=f^\ast,\qquad -\Delta w=(-f)^\ast\quad\text{in }B,
\]
where the rearrangements are signed and distribution preserving.
Then
\begin{equation*}
 \int_\Omega u^2\dd\le
 \max\left\{\int_Bv^2\dd,\int_Bw^2\dd\right\}.
\end{equation*}
\end{conjecture}

The mean-zero condition is automatic here since $\int_\Omega f\dd=0$ by integrating the Laplacian of $u \in H^2_0(\Omega)$. Hence the radial
solutions $v,w$ are also clamped, because their sources have zero
integral.

\section{Proof of Theorem~\ref{thm:gradientone} --- one-dimensional gradient comparison}
\label{sec:gradientoneproof}

The derivative follows a particularly simple rearrangement structure in
one dimension. At each level of the solution, we compare the moments
of the boundary slopes. A sharp elementary inequality for those
moments works for exponents $1 \leq p \leq 3$, after which it fails.
The integrated rearrangement $F$ has a nonzero
endpoint value when the mean is nonzero. Keeping this endpoint value
in the integrated rearrangement of $-f$ leads to the comparator selected in the theorem.

Spatial scaling reduces the proof to $|\Omega|=1$ and
$B=(-1/2,1/2)$. Use $m,f^\downarrow,F,\widetilde F$ from Section~\ref{sec:notation}.
Since $|\Omega|=1$, we have $m=\int_\Omega f\dd$ and $f^\ast(x)=f^\downarrow(2|x|)$ on $B$.
It suffices to treat $m\ge0$, since replacing $(f,u,v,w)$ by
$(-f,-u,w,v)$ gives the other case. Concavity gives $F(t)\ge mt\ge0$ on $[0,1]$.
Define
\[
 \alpha=|\{u>0\}|,\qquad \beta=|\{u<0\}|,
 \qquad \alpha+\beta\le1.
\]
With these choices, the Hardy--Littlewood
inequality reads $\int_E f\dd\le F(|E|)$ for measurable sets 
$E\subset\Omega$.

We work directly with the given source $f\in L^2(\Omega)$.
Since $u''=-f$, the solution is $C^1$ on each interval component,
and its zero extension is continuous. Thus every regular positive
level set is a compact subset of $\Omega$ consisting of finitely
many points. The one-dimensional coarea formula shows that the
critical values have measure zero, and so the level set calculations
below are valid.

For a regular value $0<\tau<\operatorname*{ess\,sup}u_+$, write
$E_\tau=\{u_+>\tau\}$, $\mu(\tau)=|E_\tau|$, and let
$\delta_1,\ldots,\delta_N$ be the positive numbers $|u'|$ at the
boundary points of $E_\tau$. Here $N\ge2$. For real $r$, put
\[
 S_r = S_r(\tau) =\sum_{j=1}^N\delta_j^r.
\]
Then 
\begin{equation*}
 S_1=\int_{E_\tau}f\dd,
 \qquad
 S_{-1}=-\mu'(\tau).
\end{equation*}
The Hardy--Littlewood inequality gives, with $t=\mu(\tau)$,
\begin{equation}\label{eq:gradientslopeHL}
 S_1\le F(t).
\end{equation}

The following moment estimate is the fulcrum of the proof:
\begin{equation}\label{eq:slopemoment}
 2^pS_{p-1}\le S_1^pS_{-1},
 \qquad 1\le p\le3.
\end{equation}
To verify it, observe that the function
$p\mapsto\log S_{p-1}$ is convex (by differentiating twice and expressing the result as a variance), whereas the logarithm of the
right side of \eqref{eq:slopemoment} is affine in $p$. Therefore it suffices to check the inequality at $p=1,3$. For $p=1$, Cauchy--Schwarz gives
\[
 S_1S_{-1}\ge N^2\ge2N = 2S_0 .
\]
For $p=3$, if $N\ge3$ then Cauchy--Schwarz and the elementary inequality $S_1^2 \geq S_2$ yield that 
\[
 S_1^3S_{-1}\ge N^2S_1^2\ge9S_1^2\ge8S_2.
\]
If $N=2$, the same inequality follows from the exact identity
\[
 S_1^3S_{-1}-8S_2
 =\frac{(\delta_1-\delta_2)^4}{\delta_1\delta_2}\ge0.
\]
This calculation proves \eqref{eq:slopemoment}.

Now apply the one-dimensional coarea formula to the gradient
integral, getting
\[
 \int_{\{u>0\}}|u'|^p\dd
 =\int_0^\infty S_{p-1}(\tau)\,d\tau.
\]
Using \eqref{eq:slopemoment}, \eqref{eq:gradientslopeHL}
and $S_{-1}(\tau)=-\mu'(\tau)$ gives
\begin{align*}
 \int_{\{u>0\}}|u'|^p\dd
 &\le \frac1{2^p}\int_0^\infty
       S_1(\tau)^pS_{-1}(\tau)\,d\tau\\
 &\le \frac1{2^p}\int_0^\infty
       F\bigl(\mu(\tau)\bigr)^{\!p}
       (-\mu'(\tau))\,d\tau\\
 &\le \frac1{2^p}\int_0^\alpha F(t)^p\,dt,
\end{align*}
where the last step is by the monotone change of variables
$t=\mu(\tau)$, with $\mu(0^+)=\alpha$.
The formula is valid even if $\mu$ has jumps or a singular part, since \(-\mu'(\tau)\,d\tau\) accounts only for the absolutely continuous part of the decrease of \(\mu\), and the integrand is nonnegative.

For $-u$ and $-f$, the same level set calculation uses the integrated rearrangement $\widetilde F(t)=F(1-t)-m$ from Section~\ref{sec:notation} in place of $F$. Although $\widetilde F$ can be negative,
at every regular nonempty level the sum of the boundary slopes is
positive and bounded above by $\widetilde F(t)$.
Thus we may use the nonnegative function
$\widetilde F_+(t)=\max\{\widetilde F(t),0\}$ in the monotone change
of variables, obtaining
\[
 \int_{\{u<0\}}|u'|^p\dd
 \le\frac1{2^p}\int_0^\beta \widetilde F_+(t)^p\,dt.
\]
Because $m\ge0$ and $F\ge0$, we have
$(F(t)-m)_+\le F(t)$. Also $\alpha\le1-\beta$, and so the positive
and reflected negative integrals occupy disjoint intervals. These
observations give
\begin{equation}\label{eq:ugradientprimitive}
\begin{split}
 \int_\Omega|u'|^p\dd
 &\le\frac1{2^p}\left[\int_0^\alpha F(t)^p\,dt
          +\int_0^\beta\widetilde F_+(t)^p\,dt\right]\\
 &=\frac1{2^p}\left[\int_0^\alpha F(t)^p\,dt
          +\int_{1-\beta}^1(F(t)-m)_+^p\,dt\right]\\
 &\le\frac1{2^p}\int_0^1F(t)^p\,dt.
\end{split}
\end{equation}

The right side of \eqref{eq:ugradientprimitive} is the
$p$-th power of the derivative norm of $v$, as the following
calculation shows.
Evenness implies $v'(0)=w'(0)=0$, and the equations
$-v^{\prime\prime}=f^*$ and $-w^{\prime\prime}=(-f)^*$ imply
\[
 v'(x)=-\frac12F(2x),\qquad
 w'(x)=-\frac12\widetilde F(2x),\qquad 0<x<\frac12.
\]
Consequently
\begin{align*}
 \int_B|v'|^p\dd =\frac1{2^p}\int_0^1F(t)^p\,dt, \qquad 
 \int_B|w'|^p\dd =\frac1{2^p}\int_0^1|F(t)-m|^p\,dt.
\end{align*}

The first identity and \eqref{eq:ugradientprimitive}
prove the comparison with $v$ for nonnegative mean.
The sign reversal described at the start of the proof gives the
comparison with $w$ for nonpositive mean.

When $m=0$, the two derivative norms coincide. For $m>0$, the function
$y\mapsto y^p-|y-m|^p$ is increasing on $[0,\infty)$ and so since $F(t)\ge mt$, we obtain
\[
 \int_0^1\bigl(F(t)^p-|F(t)-m|^p\bigr)\,dt
 \ge m^p\int_0^1\bigl(t^p-(1-t)^p\bigr)\,dt=0.
\]
Thus $\norm{v'}_{L^p(B)}\ge\norm{w'}_{L^p(B)}$ when $m>0$. The reverse ordering holds when $m<0$.

\subsection*{Where does the proof break down for $p>3$?}

The restriction on $p$ comes from the slope moment inequality
\eqref{eq:slopemoment}. For the special case of a single interval having endpoint slopes of magnitude $\delta_1=e^x$ and $\delta_2=e^{-x}$, the ratio of the left side of \eqref{eq:slopemoment} divided by its right side is
\[
 \frac{\cosh((p-1)x)}{\cosh(x)^{p+1}}
 =1+\frac{p(p-3)}2x^2+O(x^4)\qquad(x\to0) ,
\]
which is greater than $1$ for $x$ sufficiently small. Thus \eqref{eq:slopemoment} fails in general for every $p>3$.

\subsection*{Counterexamples to the gradient comparison when \texorpdfstring{$n\ge2$}{n >= 2}}

Take $p=2$. In every dimension $n\ge2$ the paired gradient comparison fails, because two concentrated sources produce two leading order 
 energies before rearrangement, whereas only one leading order energy remains after rearrangement.

For a concrete example, use the open set $\Omega$, source $f_\epsilon$, and solutions
$u_\epsilon$ and $v_\epsilon=w_\epsilon$ from
the clamped construction in Section~\ref{sec:clampedcounterexamples}, with $n\ge2$. A direct radial calculation gives
\begin{equation*}
 \frac{\int_\Omega|\nabla u_\epsilon|^2\dd}
 {\int_{B_1}|\nabla v_\epsilon|^2\dd}
 \longrightarrow2\qquad(\epsilon\downarrow0).
\end{equation*}
Thus the paired gradient comparison fails in every higher
dimension, even for smooth bounded open sets and pointwise clamped
solutions.

\section*{Funding acknowledgment}

Richard Laugesen's research was supported by grants from the
Simons Foundation (\#964018) and the National Science Foundation
(\#2246537).
Laugesen is grateful to the Laboratoire de Mathématiques, Université
Savoie Mont Blanc, for generous travel support.

Dorin Bucur was partially supported by the ANR project STOIQUES financed by
the French Agence Nationale de la Recherche (ANR-24-CE40-2216).

\section*{Acknowledgment of AI assistance}

The authors conjectured the paired comparison
inequality in Theorem~\ref{thm:planarlp} and explored it
interactively with ChatGPT and Codex. The AI system assisted
in developing and checking mathematical arguments and in constructing and numerically evaluating
examples and counterexamples.

The authors selected the questions to pursue, assessed the
resulting arguments, and take responsibility for all mathematical
statements and proofs.

\begin{quote}
\itshape
A mathematician, an editor, and an AI system walk into a bar.\\
``I conjecture this to be the optimal bar,'' says the mathematician.\\
``Only in the abstract,'' corrects the editor.\\
``Are they hallucinating?'' asks the AI.
\end{quote}


\begin{thebibliography}{99}
\bibitem{AlvinoLionsTrombetti1989}
A.~Alvino, G.~Trombetti and P.-L.~Lions,
\emph{On optimization problems with prescribed rearrangements},
Nonlinear Analysis: Theory, Methods \& Applications \textbf{13} (1989),
no.~2, 185--220.
\url{https://doi.org/10.1016/0362-546X(89)90043-6}

\bibitem{LionsAlvinoTrombetti1990}
A.~Alvino, G.~Trombetti and P.-L.~Lions,
\emph{Comparison results for elliptic and parabolic equations via
Schwarz symmetrization},
Annales de l'Institut Henri Poincar\'e C, Analyse non lin\'eaire
\textbf{7} (1990), no.~2, 37--65.
\url{https://doi.org/10.1016/S0294-1449(16)30303-1}

\bibitem{AshbaughBenguria1995}
M.~S.~Ashbaugh and R.~D.~Benguria,
\emph{On Rayleigh's conjecture for the clamped plate and its
generalization to three dimensions},
Duke Mathematical Journal \textbf{78} (1995), no.~1, 1--17.
\url{https://doi.org/10.1215/S0012-7094-95-07801-6}

\bibitem{ABLL2025}
M.~S.~Ashbaugh, D.~Bucur, R.~S.~Laugesen and R.~Leylekian,
\emph{Fourth order Saint-Venant inequalities:
maximizing compliance and mean deflection among clamped plates},
Communications in Mathematical Physics \textbf{406} (2025), article~256.
\url{https://doi.org/10.1007/s00220-025-05439-7}

\bibitem{AshbaughLaugesen1996}
M.~S.~Ashbaugh and R.~S.~Laugesen,
\emph{Fundamental tones and buckling loads of clamped plates},
Annali della Scuola Normale Superiore di Pisa, Classe di Scienze,
Serie IV, \textbf{23} (1996), no.~2, 383--402.
\url{https://numdam.org/item/ASNSP_1996_4_23_2_383_0/}

\bibitem{BaernsteinDrasinLaugesen2019}
A.~Baernstein II (with D.~Drasin and R.~S.~Laugesen),
\emph{Symmetrization in Analysis},
New Mathematical Monographs, vol.~36, Cambridge University Press,
Cambridge, 2019.
\url{https://doi.org/10.1017/9781139020244}

\bibitem{BucurLaugesenForthcoming}
D.~Bucur and R.~S.~Laugesen,
\emph{Rayleigh’s conjecture for vibrating plates: simpler planar proof and
nonlinear extensions},
in development.

\bibitem{FeroneVolpicelli2003}
A.~Ferone and R.~Volpicelli,
\emph{Minimal rearrangements of Sobolev functions: a new proof},
Annales de l'Institut Henri Poincar\'e C, Analyse non lin\'eaire
\textbf{20} (2003), no.~2, 333--339.
\url{https://doi.org/10.1016/S0294-1449(02)00012-4}

\bibitem{Kesavan2006}
S.~Kesavan,
\emph{Symmetrization \& Applications},
Series in Analysis, vol.~3, World Scientific, Hackensack, NJ, 2006.
\url{https://doi.org/10.1142/9789812773937}

\bibitem{Kristaly2020}
A.~Krist\'aly,
\emph{Fundamental tones of clamped plates in nonpositively curved spaces},
Advances in Mathematics \textbf{367} (2020), 107113.
\url{https://doi.org/10.1016/j.aim.2020.107113}

\bibitem{Kristaly2022}
A.~Krist\'aly,
\emph{Lord Rayleigh's conjecture for vibrating clamped plates in positively curved spaces},
Geometric and Functional Analysis \textbf{32} (2022), 881--937.
\url{https://doi.org/10.1007/s00039-022-00606-7}

\bibitem{Leylekian2024}
R.~Leylekian,
\emph{Des plaques vibrantes: optimisation spectrale pour l'op\'erateur
bilaplacien}, Ph.D. thesis, Aix--Marseille Universit\'e, 2024.
\url{https://hal.science/tel-05096810v1}

\bibitem{McGillivray2017}
I.~E.~McGillivray,
\emph{An unstable two-phase membrane problem and maximum flux exchange flow},
Applied Mathematics \& Optimization \textbf{75} (2017), 365--401.
\url{https://doi.org/10.1007/s00245-016-9335-7}

\bibitem{Nadirashvili1995}
N.~S.~Nadirashvili,
\emph{Rayleigh's conjecture on the principal frequency of the clamped
plate},
Archive for Rational Mechanics and Analysis \textbf{129} (1995),
no.~1, 1--10.
\url{https://doi.org/10.1007/BF00375124}

\bibitem{SaldanaTavares2018}
A.~Salda\~na and H.~Tavares,
\emph{Least energy nodal solutions of Hamiltonian elliptic systems with
Neumann boundary conditions},
Journal of Differential Equations \textbf{265} (2018), no.~12,
6127--6165.
\url{https://doi.org/10.1016/j.jde.2018.07.013}

\bibitem{Talenti1976}
G.~Talenti,
\emph{Elliptic equations and rearrangements},
Annali della Scuola Normale Superiore di Pisa,
Classe di Scienze, Serie IV, \textbf{3} (1976), no.~4, 697--718.
\url{https://numdam.org/item/ASNSP_1976_4_3_4_697_0/}

\bibitem{Talenti1981}
G.~Talenti,
\emph{On the first eigenvalue of the clamped plate},
Annali di Matematica Pura ed Applicata, Serie IV,
\textbf{129} (1981), 265--280.
\url{https://doi.org/10.1007/BF01762146}

\bibitem{Talenti1987}
G.~Talenti,
\emph{Assembling a rearrangement},
Archive for Rational Mechanics and Analysis \textbf{98} (1987),
285--293.
\url{https://doi.org/10.1007/BF00276910}

\bibitem{Talenti2016}
G.~Talenti,
\emph{The art of rearranging},
Milan Journal of Mathematics \textbf{84} (2016), no.~1, 105--157.
\url{https://doi.org/10.1007/s00032-016-0253-6}

\bibitem{VandenBergBucur2020}
M.~van den Berg and D.~Bucur,
\emph{Sign changing solutions of Poisson's equation},
Proceedings of the London Mathematical Society \textbf{121} (2020),
no.~3, 513--536.
\url{https://doi.org/10.1112/plms.12334}
\end{thebibliography}
\end{document}